\documentclass{amsart}

\usepackage{amssymb,latexsym,amsthm, amsmath}
\usepackage{graphicx}

\ifx\pdfoutput\undefined
\usepackage{hyperref}
\else
\usepackage[pdftex,colorlinks=true,linkcolor=blue,urlcolor=blue]{hyperref}
\fi

\theoremstyle{plain}
\newtheorem{theorem}{Theorem}

\newtheorem{proposition}{Proposition}

\newtheorem{lemma}{Lemma}
\theoremstyle{definition}
\newtheorem{definition}{Definition}
\newtheorem{case}{Case}
\newtheorem*{ack}{Acknowledgements}
\begin{document}
\title[Quadratic Differentials]
      {On the non-uniform hyperbolicity of the Kontsevich-Zorich cocycle for quadratic differentials} 
\author{Rodrigo Trevi\~no}
\address{Department of Mathematics\\
         The University of Maryland, College Park\\
         College Park, MD 20742}
\email{rodrigo@math.umd.edu}
\thanks{Supported by the Brin and Flagship Fellowships at the University of Maryland.}
\date{\today}
\begin{abstract}
We prove the non-uniform hyperbolicity of the Kontsevich-Zorich cocycle for a measure supported on abelian differentials which come from non-orientable quadratic differentials through a standard orienting, double cover construction. The proof uses Forni's criterion \cite{giovanni:criterion} for non-uniform hyperbolicity of the cocycle for $SL(2,\mathbb{R})$-invariant measures. We apply these results to the study of deviations in homology of typical leaves of the vertical and horizontal (non-orientable) foliations and deviations of ergodic averages.
\end{abstract}
\maketitle

It is well known that the properties of a geodesic foliation (or flow) on a flat 2-torus are completely characterized by its slope, whereas for a flat surface of higher genus the situation is far from similar. Such Riemann surface $M$ of genus greater than one with a flat metric outside finitely many singularities can be given a pair of transverse, measured foliations (in the sense of Thurston). If such foliations are orientable, Zorich \cite{zorich-leaves} detected numerically that homology classes of segments of typical leaves of the foliation deviate from the asymptotic cycle (which is defined as the limit of normalized segments of leaves) in an unprecedented way, and that the rate of deviations are given by the positive Lyapunov exponents of the Kontsevich-Zorich cocycle. Based on numerical experiments, the \emph{Kontsevich-Zorich conjecture} was formulated, which claimed that for Lebesgue-almost all classes of conformally equivalent flat metrics with orientable foliations, the exponents are all distinct and non-zero. In other words, the cocycle is non-uniformly hyperbolic and has a simple spectrum. It was also conjectured that there should be similar deviation phenomena for ergodic averages of functions in some space of functions.

The first proof of the non-uniform hyperbolicity of the Kontsevich-Zorich cocycle came from Forni \cite{forni:deviation}, but the simplicity question remained open for surfaces of genus greater than 2. The full conjecture was finally proved through methods completely different from those of Forni by Avila and Viana \cite{AvilaViana}. In \cite{forni:deviation}, a complete picture is painted on the deviations of ergodic averages along the straight line flows given by vector fields tangent to the foliations on the flat surface. The rate of divergence of such deviations are also described by all of the Lyapunov exponents of the Kontsevich-Zorich cocycle.

In this paper we study the same phenomena for the case of non-orientable foliations on flat surfaces. Although there is no vector field to speak of, we can still describe deviations of integrals of functions along leaves of the foliation. Our work has been made substantially easier by the recent criterion of Forni \cite{giovanni:criterion}, where the proof of non-uniform hyperbolicity in \cite{forni:deviation} has been condensed and generalized to apply to special $SL(2,\mathbb{R})$-invariant measures in the moduli space of abelian differentials. Note that if one has a flat surface with a non-orientable foliation, one can always pass to a double cover whereon the lift of the foliation becomes orientable. The measure on the moduli space of abelian differentials which is supported on differentials which are the pullback of non-orientable differentials is shown here to satisfy Forni's criterion. Thus most of the work is done in studying how information of the original surface is related to the information on covering surface, which is a solved problem by the works of Zorich and Forni. 

The crucial ingredient in Forni's criterion is to show there that exists a point in the support of an $SL(2,\mathbb{R})$-invariant probability measure with a completely periodic foliation whose homology classes of closed leaves span a Lagrangian subspace of the first homology space. We overcome this by a much stronger statement, showing that these special points are in fact \emph{dense} in the moduli space. We are very interested to see what the tools from generalized permutations can say to this end.

There is a canonically defined involution on the orienting double cover corresponding to the choice of orientation of the covering foliations. The involution splits the bundle on which the Kontsevich-Zorich cocycle acts into invariant and anti-invariant sub-bundles, corresponding to eigenvalues $\pm1$ of map induced by the involution. The Kontsevich-Zorich cocycle respects such splitting, defines two cocycles by its restriction to the invariant and anti-invariant sub-bundles, and thus the spectrum of the cocycle can be written as the spectrum of those two cocycles. Unlike the case for abelian differentials, \emph{the exponents which describe the deviations in homology are not the same exponents which describe the deviations of ergodic averages}, and vice-versa. Specifically, the Lyapunov exponents of the cocycle restricted to the invariant sub-bundle describe the deviations in homology of typical leaves of non-orientable foliations while the exponents of the cocycle restricted to the anti-invariant sub-bundle describe the deviations of averages of functions along leaves of non-orientable foliations. Since for any genus $g$ surface the anti-invariant sub-bundle can have arbitrarily large dimension (due to the presence of simple poles), there are non-orientable foliations on a genus $g$ surface on which the deviation of the ergodic averages along its leaves are described by arbitrarily many parameters.

Like in the original proof for abelian differentials, the proof here cannot address the question of simplicity of the Lyapunov spectrum of the cocycle. Since the restriction of the cocycle to the invariant part is equivalent to the cocycle over the moduli space of non-orientable quadratic differentials and since the anti-invariant sub-bundle describes the deviations of ergodic averages, there is no reason a-priori of why the spectrum of the cocycle over the moduli space of non-orientable quadratic differentials describes the deviations of averages of functions along leaves of non-orientable foliations defined by such quadratic differentials. Thus, unless there is some repetition of exponents across the invariant/anti-invariant division, the cocycle over the space of non-orientable quadratic differentials does not say anything about such averages. In our own numerical experiments we have found strong evidence that the spectrum of the cocycle is in fact simple.

The paper is organized as follows. In Section \ref{sec:quadratic} we review the necessary material for quadratic differentials, the double cover construction and the absolutely continuous $SL(2,\mathbb{R})$-invariant ergodic probability measure defined on each stratum of the moduli space of quadratic differentials. In Section \ref{sec:KZ} we define the Kontsevich-Zorich cocycle and state Forni's criterion for the non-uniform hyperbolicity of the cocycle. In Section \ref{sec:NUH} we show that the measure supported on abelian differentials which come from non-orientable differentials through the double cover construction satisfy Forni's criterion and thus that the Kontsevich-Zorich cocycle is non-uniformly hyperbolic with respect to that measure. In Section \ref{sec:deviations} with study the applications to deviation phenomena of homology classes and ergodic averages. Finally, in the appendix, we summarize our experimental findings of approximating numerically the Lyapunov exponents for different strata, which strongly suggest the simplicity of the cocycle.
\begin{ack}
I would like to thank my advisor, Giovanni Forni, for introducing me to this area, for suggesting this problem and for many insightful discussions during the course of this work as well as many corrections to early drafts of this paper. 
I would like to thank Rapha\"{e}l Krikorian and the Laboratoire de Probabilit\'{e}s et Mod\`{e}les Al\'{e}atoires at Universit\'{e} Paris VI, where this work was carried out, for providing excellent working conditions during my visit. I would like to thank Rafael de la Llave for many helpful discussions related to the issue of computing the Lyapunov exponents numerically.
\end{ack}
\section{Quadratic Differentials and Flat Surfaces}
\label{sec:quadratic}
Let $M$ be an orientable surface of genus $g$ and let $\Sigma_\kappa = \{p_1, \dots, p_\tau\}$ be a set of points on $M$ with $\kappa = \{n_1,\dots,n_\tau\}$, $\sum_i n_i = 4g-g$, and $n_i \in \{-1\}\cup \mathbb{N}$. $M$ is a \emph{half-translation surface} if transitions between charts on $M \backslash \Sigma_\kappa$ are given by functions of the form $\varphi(z) = \pm z + c$ for some constant $c$. On $M\backslash \Sigma_\kappa$ there is a flat metric for which the points $\Sigma_\kappa$ are singularities of order $n_i$ at $p_i$. On any such surface, we can place a pair of orthogonal foliations $\mathcal{F}^v$ and $\mathcal{F}^h$ which are defined everywhere on $M\backslash \Sigma_\kappa$ and have singularities at $\Sigma_\kappa$.

The same information is carried by a \emph{quadratic differential} on $M$. A holomorphic quadratic differential assigns to any local coordinate $z$ a quadratic form $q = \phi(z)dz^2$ where $\phi(z)$ has poles of order $n_i$ at $p_i$. If we represent it as $\phi'(w)$ with respect to another coordinate chart $w$, then it satisfies $\phi'(w) = \phi(z) (dz/dw)^2$. The foliations are then defined by integrating the distributions $\phi(z)dz^2 > 0$ and $\phi(z) dz^2<0$, respectively. In other words,
$$\mathcal{F}^v_q = \ker \mathrm{Re}\,q^{1/2} \hspace{.3in}\mbox{ and }\hspace{.3 in}\mathcal{F}^h_q = \ker \mathrm{Im}\,q^{1/2}$$
are, respectively, the \emph{vertical} and \emph{horizontal} foliations defined by a quadratic differential $q$. They are measured foliations in the sense of Thurston with respective transverse measures $|\mathrm{Re}\,q^{1/2}|$ and $|\mathrm{Im}\,q^{1/2}|$. The flat metric comes from the adapted local coordinates 
$$\zeta = \int_p^z \sqrt{\phi(w)}\, dw$$
around any point $p\in M\backslash \Sigma_\kappa$. 

If a quadratic differential is globally the square of an abelian differential, i.e., a holomorphic 1-form, then the foliations $\mathcal{F}^h_q$ and $\mathcal{F}^v_q$ are orientable and change of coordinates are given by maps of the form $\varphi(z) = z + c$. In this case we speak of a \emph{translation surface}.

Let $\mathcal{H}_g$ be the \emph{moduli space} of abelian differentials on a genus $g$ surface, which is the set of conformally equivalent classes of abelian differentials for a surface $M$ of genus $g$. The singularities in this case satisfy $\sum_i n_i = 2g-2$ and the complex dimension of this space is $2g + \tau -1$. The space $\mathcal{H}_g$ is stratified by the singularity pattern $\kappa = \{n_1,\dots,n_\tau\}$. As such, the set 
$$\mathcal{H}_\kappa = \mathcal{H}_g\cap\{\mbox{abelian differentials with singularity pattern }\kappa\}$$
 is the stratum of all abelian differentials on a genus $g$ surface with singularity pattern $\kappa = \{n_1,\dots,n_\tau\}$ and $\sum_i n_i = 2g-2$. We will interchangeably use the terms \emph{abelian differential}, \emph{quadratic differential which is a square of an abelian}, and \emph{orientable quadratic differential} since a quadratic differential $q$ with $\mathcal{F}^{v,h}_q$ orientable is necessarily the square of an abelian differential $\alpha$ and thus we can identify $q$ with $\alpha$. Note that an orientable quadratic differential has two square roots. Since they are part of the same $SL(2,\mathbb{R})$ orbit, it does not matter which square root, $+$ or $-$, we consider and thus we will by convention always pick $+$.
Thus the space of quadratic differentials which are squares of abelian is equally stratified.

The \emph{moduli space of quadratic differentials} $\mathcal{H}_g \coprod \mathcal{Q}_g$ on a Riemann surface $M$ of genus $g\geq 1$ is the quotient of the \emph{Teichmuller space of meromorphic quadratic differentials} with at most simple poles
$$\mathcal{M}_g\equiv \{\mbox{meromorphic quadratic differentials}\}/\mathrm{Diff}_0^+(M)$$
with respect to the action of the mapping class group $\Gamma_g$, where $\mathrm{Diff}_0^+$ denotes the set of orientation preserving diffeomorphisms isotopic to the identity. The subset $\mathcal{Q}_g$ denotes the set of meromorphic quadratic differentials which are \emph{not} the square of abelian differentials. These sets are equally stratified: for some singularity pattern $\kappa = \{n_1,\dots,n_\tau\}$ with $\sum_i n_i = 4g-4$, $\mathcal{Q}_\kappa$ denotes the set of quadratic differentials on a surface of genus $g$ with singularity pattern $\kappa$. Elements of $\mathcal{Q}_\kappa$ will be sometimes called \emph{non-orientable} quadratic differentials since they induce a half-translation structure on $M$, i.e., non-orientable foliations $\mathcal{F}^{v,h}_q$. Clearly it is necessary for all quadratic differentials in $\mathcal{H}_\kappa$ to have each singularity be of even order, but it is not sufficient. In fact, a result of Masur and Smillie \cite{Masur-Smillie} states that for any $\kappa=\{n_1,\dots,n_\tau\}$ with $\sum_i n_i = 4g-4$ there is a non-orientable quadratic differential $q\in \mathcal{Q}_\kappa$ with such singularity pattern with two exceptions ($\kappa = \{-1,1\}$ or $\varnothing$) in genus 1 and two exceptions ($\kappa=\{4\}$ or $\{1,2\}$) in genus two. Additionally, each stratum of $\mathcal{H}_g$ or $\mathcal{Q}_g$ is not necessarily connected. Kontsevich and Zorich \cite{KZ:connected} have achieved a complete classification of the connected components of each stratum of abelian differentials while Lanneau \cite{lanneau:connected} has classified the connected components of the strata of non-orientable differentials. The space $\mathcal{Q}_\kappa$ has complex dimension $2g + \tau - 2$.

Given any quadratic differential $q\in\mathcal{Q}_\kappa$ on a genus $g$ surface $M$ one can construct a canonical double cover $\pi_\kappa: \hat{M} \rightarrow M$ with $\hat{M}$ connected if and only if $q$ is \emph{not} the square of an abelian differential. Moreover, $\pi_\kappa^* q = \hat{\alpha}^2$, where $\hat{\alpha}$ is an abelian differential on $\hat{M}$. The construction can be summarized as follows for a non-orientable differential $q$. Let $(U_i,\phi_i)$ be an atlas for $M\backslash\Sigma_\kappa$. For any $U_i$ define $g^\pm_i(z) = \pm\sqrt{\phi_i(z)}$ on the open sets $V_i^\pm$ which are each a copy of $U_i$. The charts $\{V_i^\pm\}$ can then be glued together in a compatible way and after filling in the holes given by $\Sigma_\kappa$ we get the surface $\hat{M}$ with a quadratic differential $\hat{\alpha}^2 = \pi_\kappa^*q$. The surface $\hat{M}$ is an orienting double cover since $\mathcal{F}^{v,h}_q$ for $q\in\mathcal{Q}_g$ lifts to an orientable foliation on $\hat{M}$.

Let $\kappa$ be written as $\kappa = \{n_1,\dots, n_\nu, n_{\nu+1},\dots, n_\tau\}$ where $n_i$ is odd for $1\leq i \leq \nu$ and even for $\nu<i\leq \tau$ with $n_1\leq \dots \leq n_\nu$. Then the double cover construction gives a local embedding of  $\mathcal{Q}_\kappa$ for $\kappa = \{n_1,\dots, n_\nu, n_{\nu+1},\dots, n_\tau\}$ into $\mathcal{H}_{\hat{\kappa}}$, where
$$\hat{\kappa} = \left\{n_1+1,\dots, n_\nu + 1, \frac{1}{2}n_{\nu+1}, \frac{1}{2}n_{\nu + 1},\dots, \frac{1}{2}n_\tau, \frac{1}{2}n_\tau\right\}.$$
In the double cover construction, the preimages of the poles become marked points, the odd zeros of $q$ are critical points of $\pi_\kappa$ (ramification points) and each even singularity of $q$ has two preimages. The genus $\hat{g}$ of $\hat{M}$ can be computed by the Riemann-Hurwitz formula and satisfies $2\hat{g} = \nu+4g -2$. %

There is an involution $\sigma:\hat{M}\rightarrow\hat{M}$ mapping $\sigma:V_i^\pm\rightarrow V_i^\mp$ (that is, interchanging the points on each fiber) and clearly fixing $\pi_\kappa^{-1}\Sigma_\kappa$ as a set. Let $\hat{\Sigma}_\kappa \equiv \pi_\kappa^{-1}\Sigma_\kappa\backslash \pi_\kappa^{-1}(\{p_1,\dots,p_{\tau_{-1}}\})$, where $p_1,\dots, p_{\tau_{-1}}$ are simple poles of the quadratic differential $q$. The involution induces a splitting on the relative homology and cohomology of $\hat{M}$ into invariant and anti-invariant subspaces. Specifically, there is the following symplectic decomposition
\begin{equation}
\label{eqn:relativesplit}
H_1(\hat{M},\hat{\Sigma}_\kappa;\mathbb{R}) = H_1^+(\hat{M},\hat{\Sigma}_\kappa;\mathbb{R}) \oplus H_1^-(\hat{M},\hat{\Sigma}_\kappa;\mathbb{R})
\end{equation}
where the splitting corresponds to the eigenvalues $\pm1$ of $\sigma_*$. There is also a similar symplectic splitting in $H^1(\hat{M}, \hat{\Sigma}_\kappa;\mathbb{R})$:
$$H^1(\hat{M},\hat{\Sigma}_\kappa;\mathbb{R}) = H^1_+(\hat{M},\hat{\Sigma}_\kappa;\mathbb{R}) \oplus H^1_-(\hat{M},\hat{\Sigma}_\kappa;\mathbb{R}).$$
We will denote by $P^\pm = \frac{1}{2}(\mathrm{Id} \pm \sigma_*):H_1(\hat{M},\hat{\Sigma}_\kappa;\mathbb{Q})\rightarrow H_1^\pm(\hat{M},\hat{\Sigma}_\kappa;\mathbb{Q})$ and $P^\pm = \frac{1}{2}(\mathrm{Id} \pm \sigma^*):H^1(\hat{M},\hat{\Sigma}_\kappa;\mathbb{Q})\rightarrow H^1_\pm(\hat{M},\hat{\Sigma}_\kappa;\mathbb{Q})$ the projection to the corresponding eigenspaces in both cases.

A small neighborhood of $[\hat{\alpha}]$ in $H^1_-(\hat{M},\hat{\Sigma}_\kappa;\mathbb{C})$ gives a local coordinate chart of a regular point $q$ in $\mathcal{Q}_\kappa$. In other words, elements of $H_-^1(\hat{M},\hat{\Sigma}_\kappa;\mathbb{C})$ are abelian differentials which come from the pull-back of non-orientable quadratic differentials, $[\hat{\alpha}]\in H^1_-(\hat{M},\hat{\Sigma};\mathbb{C})$, where $\hat{\alpha} = \sqrt{\pi_\kappa^* q}$. The local charts are given by the period map $q\mapsto [\sqrt{\pi_\kappa^* q}]\in H^1_-(\hat{M},\hat{\Sigma}_\kappa;\mathbb{C})$. 

There is a canonical absolutely continuous invariant measure $\mu_\kappa$ on any stratum $\mathcal{Q}_\kappa$ of the moduli space $\mathcal{Q}_g$ defined as the Lebesgue measure on $H^1_-(\hat{M},\hat{\Sigma}_\kappa;\mathbb{C})$ normalized so that the quotient torus $H^1_-(\hat{M},\hat{\Sigma}_\kappa;\mathbb{C})/H^1_-(\hat{M},\hat{\Sigma}_\kappa;\mathbb{Z}\oplus i\mathbb{Z})$ has volume one. We remark that an analogous canonical absolutely continuous invariant measure $\nu_\kappa$ can be defined for the moduli space $\mathcal{H}_\kappa$ of squares of abelian differentials. Since the period map $q\mapsto[q^{1/2}]\in H^1(M,\Sigma_\kappa;\mathbb{C})$ gives local coordinates to $\mathcal{H}_\kappa$, it is defined in the same way and has the same properties as the measure $\mu_\kappa$ defined on strata of the moduli space of non-orientable quadratic differentials. 

The group $SL(2,\mathbb{R})$ acts on quadratic differentials $q\in(\mathcal{H}_g\coprod \mathcal{Q}_g)$ by left multiplication on the (locally defined) vector $(\mathrm{Re}\,q^{1/2},\mathrm{Im}\,q^{1/2})$. More precisely, since local coordinates are given by 
$$H^1_-(\hat{M},\hat{\Sigma}_\kappa;\mathbb{C})\cong \mathbb{R}^2\otimes H^1_-(\hat{M},\hat{\Sigma}_\kappa;\mathbb{R})$$
($H^1(M,\Sigma_\kappa;\mathbb{C})$ in the case of an orientable differential), $SL(2,\mathbb{R})$ acts on $\mathcal{Q}_\kappa$ by multiplication on the first factor. Thus, the measures $\mu_\kappa$ and $\nu_\kappa$ respectively defined on $\mathcal{Q}_\kappa$ and $\mathcal{H}_\kappa$ are $SL(2,\mathbb{R})$-invariant.

The local embedding $i_\kappa:\mathcal{Q}_\kappa\hookrightarrow \mathcal{H}_{\hat{\kappa}}$ defined by the double cover construction induces a map which maps the measure $\mu_\kappa$ to the measure 
\begin{equation}
\label{eqn:measure}
\hat{\mu}_\kappa\equiv (i_\kappa)_* \mu_\kappa
\end{equation}
on $\mathcal{H}_{\hat{\kappa}}$. Thus, the measure $\hat{\mu}_\kappa$ is singular with respect to $\nu_{\hat{\kappa}}$ since the support of $\hat{\mu}_\kappa$ is the sub-variety of $\mathcal{Q}_{\hat{\kappa}}$ which is the preimage of the subspace $H^1_-(\hat{M},\hat{\Sigma}_\kappa;\mathbb{C})$ under the period map. The measure (\ref{eqn:measure}) is clearly $SL(2,\mathbb{R})$-invariant.

\section{The Kontsevich-Zorich Cocycle}
\label{sec:KZ}
The action of diagonal subgroup
$$g_t \equiv \left\langle \left( \begin{array}{cc}
e^t & 0 \\
0 & e^{-t} \end{array} \right) : t\in\mathbb{R}\right\rangle\leq SL(2,\mathbb{R})$$
on $\mathcal{H}_\kappa$ or $\mathcal{Q}_\kappa$ is the \emph{Teichmuller flow} and plays a central role in the study of quadratic differentials. It is was proved by Masur \cite{Masur:IET} for the principal stratum $\kappa = \{1,\dots,1\}$ and then for any stratum by Veech \cite{veech-teich} that the Teichmuller flow acts ergodically on each connected component of a stratum with respect to the measure $\mu_\kappa$ (respectively, $\nu_\kappa$) when restricted to a hypersurface $\mathcal{Q}_\kappa^{(A)}\subset \mathcal{Q}_\kappa$ of quadratic differentials on a surface of area $A$ (respectively, the hypersurface $\mathcal{H}_\kappa^{(A)}\subset \mathcal{H}_\kappa$ of abelian differentials of norm $A$) and that the measure $\mu_\kappa^{(A)}\equiv \mu_\kappa|_{\mathcal{Q}_\kappa^{(A)}}$ (respectively, $\nu_\kappa^{(A)}\equiv \nu_\kappa|_{\mathcal{H}_\kappa^{(A)}}$) is finite. 

The Teichmuller flow $g_t$ admits two invariant foliations $\mathcal{W}^\pm$ on $\mathcal{H}_g$. For an abelian differential $\alpha\in\mathcal{H}_g$ , the foliations are locally defined by
\begin{eqnarray*}
  \mathcal{W}^+(\alpha)&=& \{ \alpha'\in\mathcal{H}_g : \mathrm{Im}\,\alpha' \in \mathbb{R}^+\cdot\mathrm{Im}\,\alpha\} = \{\alpha'\in\mathcal{H}_g:\mathcal{F}_{\alpha'}^h = [\mathcal{F}_\alpha^h]\} \\
\mathcal{W}^-(\alpha) &=& \{ \alpha'\in\mathcal{H}_g : \mathrm{Re}\,\alpha' \in \mathbb{R}^+\cdot\mathrm{Re}\,\alpha\} = \{\alpha'\in\mathcal{H}_g:\mathcal{F}_{\alpha'}^v = [\mathcal{F}_\alpha^v]\}.
\end{eqnarray*}
Let $\mathcal{W}_\kappa^\pm(\alpha)$ be the intersection of $\mathcal{W}^\pm(\alpha)$ with the stratum $\mathcal{H}_\kappa$. For any open set $\mathcal{U}\subset\mathcal{H}_\kappa$, define the local, invariant foliations $\mathcal{W}_{\mathcal{U}}^\pm$ as the unique, connected component of the intersection $\mathcal{W}^\pm_\kappa(\alpha)\cap\mathcal{U}$ which contains the abelian differential $\alpha\in\mathcal{U}$.

\subsection{Definition of the Cocycle}
Let $\mathcal{M}_g$ be the Teichmuller space of meromorphic quadratic differentials on a Riemann surface $M$ of genus $g>1$. The Kontsevich-Zorich cocycle $G_t$, introduced in \cite{kontsevich:hodge}, is the quotient cocycle, with respect to the mapping class group $\Gamma_g$, of the trivial cocycle
$$g_t\times \mathrm{id}:\mathcal{M}_g\times H^1(M;\mathbb{R})\longrightarrow \mathcal{M}_g\times H^1(M;\mathbb{R})$$
acting on the orbifold vector bundle
$$\mathcal{H}^1_g(M;\mathbb{R})\equiv (\mathcal{M}_g\times H^1(M;\mathbb{R}))/\Gamma_g$$
over the moduli space $Q_g\equiv (\mathcal{H}_g\coprod\mathcal{Q}_g) = \mathcal{M}_g/\Gamma_g$ of meromorphic quadratic differentials. Note that we can identify fibers of close points using the Gauss-Manin connection. The projection of the cocycle $G_t$ coincides with the Teichmuller flow $g_t$ on the moduli space $Q_g$. 

By the Oseledets Multiplicative Ergodic Theorem for linear cocycles \cite{Katok:book}, for a $g_t$-invariant probability measure $\mu$ supported on some stratum of $Q_g$ there is a decomposition $\mu$-almost everywhere of the cohomology bundle $H^1_q(M;\mathbb{R}) = E^+(q)\oplus E^-(q)\oplus E_0(q)$ where 
\begin{equation}
\label{eqn:decomp}
E^\pm(q) = E_1^\pm(q)\oplus\cdots\oplus E_{s^\pm}^\pm(q)
\end{equation}
and Lyapunov exponents $ \lambda_1^+>\dots>\lambda_{s^+}^+>0>\lambda_1^->\dots>\lambda_{s^-}^-$ which describe the exponential rate of expansion and contraction of elements in such sub-bundles under $G_t$. Elements of $E_0$ have zero exponential expansion or contraction. The dimension of each sub-bundle $E_i^\pm$ in (\ref{eqn:decomp}) is exactly the multiplicity of $\lambda_i^\pm$.

It follows from the fact that $G_t$ is a symplectic cocycle that the Lyapunov spectrum of the cocycle $G_t$, with respect to any $g_t$-invariant ergodic probability measure, is symmetric. In other words, if $\lambda$ is a Lyapunov exponent of $G_t$, so is $-\lambda$ and $\dim E^+ = \dim E^-$. Thus, the Lyapunov exponents for the Kontsevich-Zorich cocycle satisfy 
\begin{equation}
\label{eqn:exponents}
1 = \lambda_1 \geq \lambda_2 \geq \dots \geq \lambda_g \geq 0 \geq -\lambda_g = \lambda_{g+1}\geq\dots\geq\lambda_{2g-1}\geq \lambda_{2g} = -1.
\end{equation}

Since the period map identifies the tangent space of $Q_g$ to the cohomology space, there is a relationship between the Lyapunov exponents of the Kontsevich-Zorich cocycle and those of the tangent cocycle of the Teichmuller flow. Since we can express the local trivialization of the tangent bundle as $T Q_\kappa = Q_\kappa \times H^1(M,\hat{\Sigma}_\kappa;\mathbb{C})$ ($Q_\kappa \times H^1_-(M,\Sigma_\kappa;\mathbb{C})$ when $G_t$ acts on strata of non-orientable differentials), then by the isomorphism of the vector bundles
$$\mathcal{H}_\kappa^1(M,\mathbb{C})\cong \mathbb{C} \otimes H^1(M;\mathbb{R}) \cong \mathbb{R}^2\otimes H^1(M;\mathbb{R})$$
induced by the isomorphism on each fiber, the projection of $T g_t$ to the absolute cohomology can be expressed in terms of the Kontsevich-Zorich cocycle as
$$T g_t =\left( \begin{array}{cc}
e^t & 0 \\
0 & e^{-t} \end{array} \right) \otimes G_t \hspace{.3in}\mbox{acting on}\hspace{.3in}\mathbb{R}^2\otimes H^1(M;\mathbb{R}).$$

Thus, the Lyapunov exponents of the Teichmuller flow with respect to the canonical, absolutely continuous measures $\mu_\kappa$ or $\nu_\kappa$ can be written as
\begin{eqnarray*}
  2 &\geq& (1+\lambda_2)\geq\dots\geq (1+\lambda_g)\geq \overbrace{1=\dots=1}^{\tau -1} \geq (1-\lambda_g) \\ \nonumber
    &\geq& \dots \geq (1-\lambda_2)\geq 0\geq -(1-\lambda_2)\geq\dots\geq-(1-\lambda_g) \\
    &\geq&  \underbrace{-1=\dots=-1}_{\tau-1}\geq-(1+\lambda_g)\geq\dots\geq-(1+\lambda_2)\geq -2. \nonumber
\end{eqnarray*}
where the $\tau-1$ trivial exponents come from cycles relative to $\Sigma_\kappa$.

The trivial exponents of the tangent cocycle $T g_t$ are neglected by $G_t$ since the bundle $\mathcal{H}_g^1$ neglects cocycles in $H^1(M,\Sigma_\kappa;\mathbb{C})$ which are dual to cycles relative to $\Sigma_\kappa$, from which we get such trivial exponents. The non-uniform hyperbolicity of the tangent cocycle for the Teichmuller flow is equivalent to the spectral gap of the Kontsevich-Zorich cocycle, i.e., that $\lambda_1>\lambda_2$. This was proved by Veech \cite{veech-teich} for the canonical measure and then by Forni in \cite{forni:deviation} for any Teichmuller invariant ergodic probability measure $\mu$ in $\mathcal{H}_g$.

Let $\hat{q}=i_\kappa(q)\in\mathcal{H}_{\hat{\kappa}}$ be an orientable quadratic differential which is obtained by the double cover construction. The splitting $H^1(\hat{M};\mathbb{R}) = H^1_+\oplus H^1_-$ is equivariant with respect to the Gauss-Manin connection. Since both $H^1_+$ and $H^1_-$ are symplectic subspaces, the restriction of the Kontsevich-Zorich cocycle to either the invariant or anti-invariant sub-bundles defines another symplectic cocycle. Thus we get symmetric Lyapunov spectra
$$\lambda_1^+ \geq \lambda_2^+ \geq \dots \geq \lambda_g^+ \geq 0 \geq -\lambda_g^+ = \lambda_{g+1}^+\geq\dots\geq \lambda_{2g}^+$$
and
$$\lambda_1^- \geq \lambda_2^- \geq \dots \geq \lambda_{g+n-1}^- \geq 0 \geq -\lambda_{g+n-1}^- = \lambda_{g+n}^-\geq\dots\geq \lambda_{2g+2n-2}^-$$
which are, respectively, the Lyapunov exponents of the symplectic cocycles of the invariant and anti-invariant sub-bundles.

It follows from the double cover construction that the action of $g_t$ commutes with $i_\kappa$. Moreover, since $\pi^*_\kappa$ is an isomorphism between $H^1(M;\mathbb{R})$ and $H^1_+(\hat{M};\mathbb{R})$,
\begin{equation}
\label{eqn:commute}
(i_\kappa\times\pi^*_\kappa)\circ(g_t|_{\mathcal{Q}_\kappa}\times\mathrm{id}) = (g_t|_{\mathcal{H}_{\hat{\kappa}}}\times\mathrm{id})\circ(i_\kappa\times\pi^*_\kappa),
\end{equation}
and thus the Lyapunov spectrum of the Kontsevich-Zorich cocycle on the bundle over $i_\kappa(\mathcal{Q}_\kappa)$ restricted to the invariant sub-bundle is the same as the Lyapunov spectrum of the Kontsevich-Zorich cocycle on the bundle over $\mathcal{Q}_\kappa$.

\subsection{A Criterion for Non-Uniform Hyperbolicity}

The non-uniform hyperbolicity of the Kontsevich-Zorich cocycle for the canonical, absolutely continuous measure on $\mathcal{H}_\kappa$ was first proved by Forni in \cite{forni:deviation}. Recently, the proof of such result has been generalized in \cite{giovanni:criterion} to apply to \emph{any} $SL(2,\mathbb{R})$-invariant ergodic probability measure on $\mathcal{H}_\kappa$ which have special points in their support. In this section we review the necessary material to state Forni's criterion.

\begin{definition}
An open set $\mathcal{U}\subset\mathcal{H}_\kappa$ is of \emph{product type} if for any $(\omega^+,\omega^-)\in\mathcal{U}\times\mathcal{U}$ there is an abelian differential $\omega\in\mathcal{U}$ and an open interval $(a,b)\subset\mathbb{R}$ such that
$$\mathcal{W}^+_\mathcal{U}(\omega^+)\cap\mathcal{W}^-_\mathcal{U}(\omega^-) = \bigcup_{t=a}^b\{g_t(\omega)\}.$$
\end{definition}
Define for an open subset $\mathcal{U}\subset\mathcal{H}_\kappa$ of product type and any subset $\Omega\subset\mathcal{U}$, 
$$\mathcal{W}^\pm_\mathcal{U}(\Omega)\equiv \bigcup_{\omega\in\Omega}\mathcal{W}^\pm_\mathcal{U}(\omega).$$

\begin{definition}
A Teichmuller-invariant measure $\mu$ supported on $\mathcal{H}_\kappa$ has \emph{product structure} on an open subset $\mathcal{U}\subset\mathcal{H}_\kappa$ of product type if for any two Borel subsets $\Omega^\pm\subset\mathcal{U}$,
$$\mu(\Omega^-)\neq 0\mbox{ and }\mu(\Omega^+)\neq 0 \hspace{.3 in}\mbox{ implies } \hspace{.3 in}\mu\left(\mathcal{W}^+_\mathcal{U}(\Omega^+)\cap\mathcal{W}^-_\mathcal{U}(\Omega^-)\right)\neq 0.$$
A Teichmuller-invariant measure $\mu$ on $\mathcal{H}_\kappa$ has \emph{local product structure} if \emph{every} abelian differential $\omega\in\mathcal{H}_\kappa$ has an open neighborhood $\mathcal{U}_\omega\subset\mathcal{H}_\kappa$ of product type on which $\mu$ has a product structure.
\end{definition}

\begin{definition}
The \emph{homological dimension} of a completely periodic measured foliation $\mathcal{F}$ on an orientable surface $M$ of genus $g>1$ is the dimension of the isotropic subspace $\mathcal{L}(\mathcal{F})\subset H_1(M;\mathbb{R})$ generated by the homology classes of closed leaves of the foliation $\mathcal{F}$. A completely periodic measured foliation $\mathcal{F}$ is \emph{Lagrangian} if $\dim\mathcal{L}(\mathcal{F}) = g$, that is, if the subspace in $H_1(M;\mathbb{R})$ generated by classes of closed leaves of the foliation is a Lagrangian subspace with respect to the intersection form.
\end{definition}

A periodic measured foliation is Lagrangian if and only if it has $g$ distinct leaves $\gamma_1,\dots,\gamma_g$ such that $\tilde{M} = M\backslash(\gamma_1\cup\cdots\cup\gamma_g)$ is homeomorphic to a sphere minus $2g$ paired, disjoint disks.

\begin{definition}
\label{def:cuspidal}
A Teichmuller-invariant probability measure on a stratum $\mathcal{H}_\kappa$ is \emph{cuspidal} if it has local product structure and its support contains a holomorphic differential with a completely periodic horizontal or vertical foliation. The \emph{homological dimension} of a Teichmuller-invariant measure is the maximal homological dimension of a completely periodic vertical or horizontal foliation of a holomorphic differential in its support. A Teichmuller-invariant probability measure is \emph{Lagrangian} if it has maximal homological dimension, i.e., its support contains a holomorphic differential whose vertical or horizontal foliation is Lagrangian.
\end{definition}

As far as the author is aware, all known $SL(2,\mathbb{R})$-invariant measures on $\mathcal{H}_g$ (and in particular the measure (\ref{eqn:measure})) are cuspidal. We can now state Forni's criterion for the non-uniform hyperbolicity of the Kontsevich-Zorich cocycle with respect to some $SL(2,\mathbb{R})$-invariant measure.

\begin{theorem}[Forni's Criterion \cite{giovanni:criterion}]
\label{thm:criterion}
Let $\mu$ be an $SL(2,\mathbb{R})$-invariant ergodic probability measure on a stratum $\mathcal{H}_\kappa\subset\mathcal{H}_g$ of the moduli space of abelian differentials. If $\mu$ is cuspidal Lagrangian, the Kontsevich-Zorich cocycle is non-uniformly hyperbolic $\mu$-almost everywhere. The Lyapunov exponents $\lambda_1^\mu\geq\cdots\geq\lambda_{2g}^\mu$ form a symmetric subset of the real line in the following way:
$$1=\lambda_1^\mu>\lambda_2^\mu\geq\cdots\geq\lambda_g^\mu> 0 > \lambda_{g+1}^\mu = -\lambda_g^\mu \geq \cdots\geq \lambda_{2g-1}^\mu = -\lambda_2^\mu > \lambda_{2g}^\mu = -1.$$
\end{theorem}

The spectral gap $\lambda_1^\mu >\lambda_2^\mu$ is an easier result than the entire proof of non-uniform hyperbolicity. In fact, in \cite{forni:deviation} the spectral gap was proved for \emph{any} $g_t$-invariant probability measure. It follows from this result that both $E^+_1(q)$ and $E^-_{2g}(q)$ in the decomposition (\ref{eqn:decomp}) are one-dimensional. In fact, for an Oseledets-regular point $q\in\mathcal{H}_\kappa$, $E^+_1(q) = [\mathrm{Re}\,q^{1/2}]\cdot\mathbb{R}$ and $E^-_{2g}(q) = [\mathrm{Im}\,q^{1/2}]\cdot\mathbb{R}$, and their dual bundles (in the sense of Poincar\'{e} duality) in $H_1(M;\mathbb{R})$ are generated, respectively, by the Schwartzman asymptotic cycles (which will be defined in section \ref{sec:deviations}) for the horizontal and vertical foliations, $\mathcal{F}^{v,h}_q$.

\section{Non-Uniform Hyperbolicity for Quadratic Differentials}
\label{sec:NUH}
In this section we apply Forni's criterion (Theorem \ref{thm:criterion}) to the $SL(2,\mathbb{R})$-invariant measure (\ref{eqn:measure}) on $\mathcal{H}_{\hat{\kappa}}$ coming from non-orientable quadratic differentials by the double cover construction detailed in section \ref{sec:quadratic}. The non-trivial property to show is that the support of such measure in every stratum contains a completely periodic quadratic differential $q$ on $M$ whose vertical or horizontal foliation lifts to a Lagrangian foliation on $\hat{M}$, since for any surface $M$ of genus $g$, the anti-invariant space $H_1^-(\hat{M};\mathbb{R})$ can have arbitrarily large dimension. In this section we will prove a much stronger statement, Proposition \ref{prop:density}, which states that such quadratic differentials are dense in every stratum $\mathcal{Q}_\kappa$, which will suffice in order to apply Theorem \ref{thm:criterion}.

Following \cite[\S 4.1]{lanneau:parity}, we make some remarks about the structure of $\pi_\kappa:\hat{M}\rightarrow M$ and the canonical basis on homology one can construct from it. Note that 
$$\pi_\kappa:\hat{M}\backslash\{\mbox{ramification points}\}\rightarrow M\backslash\{\mbox{odd singularities}\}$$
is a regular covering space with group of deck transformations $\mathbb{Z}_2$. As such, and denoting $\dot{M} = M\backslash\{\mbox{odd singularities}\}$, the monodromy representation $\pi_1(\dot{M})\rightarrow \mathbb{Z}_2$ factors through $H_1(\dot{M};\mathbb{Z})$ (and even through $H_1(\dot{M};\mathbb{Z}_2)$) since $\mathbb{Z}_2$ is Abelian. Let $m:H_1(\dot{M};\mathbb{Z}_2)\rightarrow \mathbb{Z}_2$ denote such map. Starting with a standard symplectic basis $\{a_1,b_1,\dots,a_g,b_g\}$ for $H_1(M;\mathbb{Z}_2)$ with $a_i\cap b_i = 1$ and all other intersections zero, it is possible to construct the following (symplectic) basis on $H_1(\hat{M};\mathbb{Z})$, using that $[\gamma]\in\ker(m)$ if and only if the loop $\gamma$ lifts to two loops on $\hat{M}$.

Suppose that $M$ has no singularities of odd degree. In this case $\pi_\kappa:\hat{M}\rightarrow M$ is a regular covering space and as such $\sigma$ has no fixed points and the holonomy of a curve depends only on its homology class. Starting with a standard symplectic basis $\{\bar{a}_1,\bar{b}_1,\dots,\bar{a}_g,\bar{b}_g\}$ of $H_1(M;\mathbb{Z})$ we can make a change of basis to obtain a ``nice'' basis of $H_1(\hat{M};\mathbb{Z})$. By assumption, $q$ is not the square of an Abelian differential, so there is at least one cycle of our symplectic basis with non-trivial monodromy, which we can assume is $\bar{b}_g$. For $1\leq i < g$, let $a_i = \bar{a}_i + \bar{b}_g$ if $m(\bar{a}_i) = 1$ and otherwise $a_i = \bar{a}_i$, and construct $b_i$ in a similar way. Then any loop $\gamma_{a_i}$ or $\gamma_{b_i}$ representing the new basis $\{a_i,\,b_i\}$ lifts to two disjoint loops $\gamma^\pm_{a_i}$ and $\gamma^\pm_{b_i}$ for $1\leq i < g$ with $[\gamma^\pm_{a_i}]=a^\pm_i$ and $[\gamma^\pm_{b_i}]=b^\pm_i$.  We can assign the labels $\pm$ such that $a_i^+\cap b_i^+ = a_i^-\cap b_i^- = 1$ and all other intersections are zero for $1\leq i < g$. Because of the prescribed symplectic structure, $P^\pm a_i^+ \neq 0 \neq P^\pm b_i^+$ for $1\leq i <g$ and moreover they span a symplectic subspace of $H_1(\hat{M};\mathbb{Q})$ of dimension $4g-4$ (codimension 2).

Let $b^+_g$ be homology class of a lift of a curve representing $b_g$ on $M$ and similarly for a lift $a^+_g$ of $a_g$, independent of the value of $m(a_g)$. Then
$$\{a_1^+,b_1^+,a_1^-,b_1^-,\dots, a_{g-1}^+,b_{g-1}^+,a_{g-1}^-,b_{g-1}^-,a^+_g,b^+_g\}$$
is a basis of $H_1(\hat{M};\mathbb{Z})$. Moreover we have $a_i^- = \sigma_*a_i^+$ and $b_i^- = \sigma_*b_i^+$ for $1\leq i < g$, and $\sigma_*b^+_g = b^+_g$. Furthermore we have $\pi_{\kappa*}a^\pm_i = a_i$ and $\pi_{\kappa*}b^\pm_i = b_i$.

The cycles on $H_1(\hat{M};\mathbb{Z})$ which come through modified cycles on $H_1(M;\mathbb{Z})$ can be modified by subtracting $b_g^+$ to give a symplectic basis for $H_1(\hat{M};\mathbb{Z})$, which we can explicitly write in terms of the invariant and anti-invariant subspaces in homology:
\begin{eqnarray*}
H_1^+(\hat{M};\mathbb{Q}) &=& \langle P^+a_1^+,P^+b_1^+,\dots,P^+a^+_g,P^+b^+_g\rangle \\
H_1^-(\hat{M};\mathbb{Q}) &=& \langle P^-a_1^-,P^-b_1^-,\dots,P^-a^-_{g-1},P^-b^-_{g-1}\rangle
\end{eqnarray*}
In these coordinates, $P^+a^+_i\cap P^+b^+_i = P^-a^-_i \cap P^-b^-_i \neq 0$ and all other intersections are zero. Thus $H_1^+$ and $H_1^-$ are symplectically orthogonal.

Suppose that $M$ has some singularities of odd order, which by necessity has to be an even number of them, $2n$, and label the odd singularities $p_1,\dots, p_{2n}$. Consider a standard symplectic basis $a_1,b_1,\dots,a_g,b_g$ of $H_1(M;\mathbb{R})$. Note that two loops representing homology classes can be different in $H_1(\dot{M};\mathbb{Z})$ while being homologous in $H_1(M;\mathbb{Z})$. This happens, for example, when the loops have different monodromy. Thus any loop representing a basis element of $H_1(M;\mathbb{Z})$ with non-trivial monodromy can be modified slightly to change its monodromy while staying in the same homology class. This is done by ``taking a detour'' to go around an odd singularity, say $p_{2n}$. By making such modifications to representatives of $a_i$ and $b_i$ we can suppose that every loop representing a basis element of $H_1(M;\mathbb{Z})$ lifts to two loops on $\hat{M}$, $\gamma_{a^\pm}$ and $\gamma_{b^\pm}$ with $[\gamma_{a^\pm_i}]=a^\pm_i$ and $[\gamma_{b^\pm_i}]=b^\pm_i$. By considering the intersections of curves representing the basis of $H_1(M)$ and their lifts, we can assign the $\pm$ labels to the lifts so that we get a collection of cycles in $H_1(\hat{M};\mathbb{Q})$
\begin{equation}
\label{eqn:basispoles}
\{a_1^+,b_1^+,a_1^-,b_1^-,\dots, a_{g}^+,b_{g}^+,a_g^-,b_g^-\}
\end{equation}
such that $a_i^+\cap b_i^+ = a_i^-\cap b_i^- = 1$ for $i\leq i \leq g$ and all other intersections are zero. Moreover we have $a_i^- = \sigma_*a_i^+$ and $b_i^- = \sigma_*b_i^+$ for $1\leq i \leq g$. Furthermore we have $\pi_{\kappa*}a^\pm_i = a_i$ and $\pi_{*\kappa}b^\pm_i = b_i$. Because of the prescribed symplectic structure, $P^\pm a^+_i \neq 0 \neq P^\pm b^+_i$ for $1\leq i \leq g$ and these cycles span a $4g$-dimensional symplectic subspace of $H_1(\hat{M};\mathbb{Q})$. Thus we can explicitly write the basis for the invariant and anti-invariant subspaces in homology:
\begin{eqnarray}
\label{splitting}
H_1^+(\hat{M};\mathbb{Q}) &=& \langle P^+a_1^+,P^+b_1^+,\dots,P^+a^+_g,P^+b^+_g\rangle \nonumber \\
H_1^-(\hat{M};\mathbb{Q}) &=& \langle P^-a_1^-,P^-b_1^-,\dots,P^-a^-_g,P^-b^-_g\rangle
\end{eqnarray}
with the corresponding intersections, making them symplectically orthogonal. By the Riemann-Hurwitz formula, $\dim H_1^-(\hat{M},\mathbb{R}) = 2g + 2n -2$, so in the case of $n=1$ we have constructed a basis for the homology of the covering surface. For $n>1$, the other $2n - 2$ cycles on $\hat{M}$ which are basis elements of $H_1(\hat{M};\mathbb{R})$ are constructed in a way reminiscent of the way one constructs basis elements on a hyperelliptic surface. 

Consider a series of paths $l_i$ joining $p_i$ to $p_{i+1}$ for $1\leq i \leq 2n-2$. We can chose these paths so that they have no intersection with the cycles $a_i$ or $b_i$ and that the line $\bigcup_{i = 1}^{2n-2} l_i$ does not have self intersections. For $\varepsilon$ sufficiently small, take an $\varepsilon$-tubular neighborhood $E_i$ of $l_i$ and consider the oriented boundary $\partial E_i$ which we can identify with a cycle $\bar{c}_i$. This cycle clearly has trivial monodromy and, as such, lifts to two different paths on $\hat{M}$. Pick one of these and label it $c_i$. Thus we get the cycles $c_1,\dots,c_{2n-2}$ on $\hat{M}$ with $c_j\cap c_{j+1} = 1$ for $1\leq j \leq 2n -3$ and $\sigma_* c_j = -c_j$. Let $\mathcal{C}\subset H_1(\hat{M};\mathbb{Z})$ be the subspace spanned by the cycles $c_i$. This space is symplectically orthogonal to the subspaces spanned by $P^\pm a^\pm_i$ and $P^\pm b^\pm_i$. The subspace $\mathcal{C}$ can be thought of absolute homology classes of the covering surface which are represented by lifts of curves which are homologous to zero. We will denote by $P^\mathcal{C}:H_1(\hat{M};\mathbb{R})\rightarrow \mathcal{C}$ the projection of a cycle to $\mathcal{C}$.

For the case when $q$ has at least two odd singularities, we adopt from now on the following notation. Let $H_1^-(\hat{M};\mathbb{Z}) = \hat{H}_1^-(\hat{M};\mathbb{Z})\oplus\mathcal{C}$ be the anti-invariant eigenspace, i.e., the projection $P^-H_1(\hat{M};\mathbb{Z})$. Then we can write the homology of the covering surface, which represents the (symplectic) orthogonal splitting, as:
$$H_1(\hat{M};\mathbb{R}) = H_1^+(\hat{M};\mathbb{R})\oplus \hat{H}_1^-(\hat{M};\mathbb{R})\oplus \mathcal{C}.$$
Similarly, there is a splitting in cohomology:
$$H^1(\hat{M};\mathbb{R}) = H^1_+(\hat{M};\mathbb{R})\oplus \hat{H}^1_-(\hat{M};\mathbb{R})\oplus \mathcal{C}^*.$$
Note that when $n>1$, $\hat{H}_1^-(\hat{M};\mathbb{R})$ is \emph{not} the entire anti-invariant eigenspace, but the projection to the negative eigenspace of the cycles on $\hat{M}$ which come from basis elements of $H_1(M;\mathbb{Z})$. 

\begin{definition}
A measured foliation $\mathcal{F}$ on a compact surface is called \emph{periodic} if the set of non-closed leaves has measure zero. A quadratic differential whose horizontal foliation is periodic is called a \emph{periodic quadratic differential}. A \emph{saddle connection} is a leaf of the foliation joining two singularities.
\end{definition}

In the literature, periodic quadratic differentials also go by the name of \emph{Strebel} quadratic differentials.

We now relate structure of periodic foliations induced by quadratic differentials to the above discussion of the relationship between the homology of the half-translation surface $M$ carrying a quadratic differential and its orienting double cover $\hat{M}$. By removing saddle connections and singularities, a half-translation surface carrying a periodic quadratic differential $q$ decomposes $M$ into the disjoint union of cylinders $\{c_q^1,\dots, c_q^k\}$ composed of closed leaves of the foliation. Each cylinder $c_q^i$ has a waistcurve $|a_q^i|$ whose homology class $a_q^i = [|a_q^i|]$ represents the homology class of all other closed leaves in $c_q^i$.

\begin{lemma}
\label{lem:monodromy}
Let $q$ be a periodic quadratic differential on $M$. If $0\neq a_q^i\in H_1(M;\mathbb{Z})$, then $m(a_q^i) = 0$, i.e., the lift of $c_q^i$ consists of two cylinders on $\hat{M}$.
\end{lemma}
\begin{proof}
Suppose $a_q^i \neq 0$ and $c_q^i$ lifts to one cylinder $\hat{c}_q^i$ and let $\alpha = \sqrt{\pi_\kappa^* q}$. The involution $\sigma$ maps $\hat{c}_q^i$ to itself and, since $\sigma^* \alpha = -\alpha$, it reverses the orientation of its waistcurve. By the Lefschetz fixed point theorem, there is at least one fixed point inside $\hat{c}_q^i$. Since $\sigma$ is an isometric involution, it follows that there are exactly two fixed points. Since the only fixed points of $\sigma$ are preimages of zeros of odd order and there are no zeros of $\alpha$ in the interior of cylinders of $\hat{M}$, then the two fixed points are preimages of poles of $q$. If $q$ did not have poles to begin with, we have reached a contradiction. Furthermore, if $q$ has poles, any closed curve sufficiently close and parallel to the waistcurve of $\hat{c}_q^i$ projects to a curve, which is homologous to $|a_q^i|$, going around a saddle connection between two poles on $M$, and thus homologous to zero, which contradicts the assumption that $0\neq a_q^i\in H_1(M;\mathbb{Z})$.
\end{proof}

For any measured foliation $\mathcal{F}_q$ on $M$, denote by $\hat{\mathcal{F}}_q$ the measured foliation given by $\mathcal{F}_{\pi_\kappa^*q}$ on $\hat{M}$, i.e., the lift of $\mathcal{F}_q$ to $\hat{M}$. As such, we have that $\mathcal{F}_q$ is periodic if and only if $\hat{\mathcal{F}}_q$ is periodic. Let $\alpha$ be an Abelian differential on a translation surface $M$ which, for the next lemma, we do not assume is the pullback of a quadratic differential.  Let $S_{\alpha}$ be the union of all saddle connections in the periodic foliation given by a holomorphic 1-form $\alpha$. By convention, we also assume the singularities of $\alpha$ are contained in $S_\alpha$. Then $M\backslash S_\alpha$ is a disjoint union of cylinders $c_\alpha^1,\dots, c_\alpha^s$. 

\begin{lemma}
\label{lem:intersection}
Let $\alpha$ be an Abelian differential on a translation surface $M$ whose horizontal foliation is periodic with cylinders $\{c_\alpha^1,\dots,c_\alpha^s \}$ with respective waistcurves $\{|a_\alpha^1|,\dots,|a_\alpha^s|\}$ and heights $\{h_\alpha^i \}$. Let $\gamma:[0,1]\longrightarrow M$ be a simple curve with $\gamma(0),\gamma(1)\in S_\alpha$. Then
\begin{equation}
\label{eqn:intersection}
\int_{\gamma}\alpha = \sum_{i=1}^s h_\alpha^i \left([\gamma_i]\cap a_\alpha^i\right),
\end{equation}
where $[\gamma_i] \equiv [\gamma \cap c_\alpha^i] \in H_1(c_\alpha^i,\partial c_\alpha^i;\mathbb{Z})$.
\end{lemma}
\begin{proof}
Since $M$ decomposes into cylinders, 
$$\int_\gamma \alpha = \sum_{i=1}^s \int_{\gamma\cap c_\alpha^i}\alpha.$$
Moreover, in each cylinder $\alpha$ can be written in local coordinates as $dy_i$. Thus
$$\int_\gamma \alpha = \sum_{i=1}^s \int_{\gamma\cap c_\alpha^i} dy_i,$$
from which (\ref{eqn:intersection}) follows.
\end{proof}
Note that in Lemma \ref{lem:intersection} we did not require $\gamma$ to be closed. The lemma thus yields information of the intersection properties of curves $\gamma$ with waistcurves of cylinders of $M$ defined by a periodic Abelian differential. It follows that any periodic $\hat{\mathcal{F}}_q$ is given by a holomorphic 1-form $\alpha$ with the property that $P^{-1}\alpha = \sum_{i=1}^s h^i_\alpha a^i_\alpha$, where $h^i_\alpha > 0$ is the height of the cylinder $c_\alpha^i$, $a_\alpha^i$ is the homology class represented by its oriented waistcurve $|a_\alpha^i|$ (with respect to the orientation of the foliation), and $P$ is the (symplectic) isomorphism given by Poincar\'e duality. 

Let $I(q)$ and $I(\alpha)$ denote the maximal isotropic subspaces of $H_1(M;\mathbb{Q})$ and $H_1(\hat{M};\mathbb{Q})$, respectively, spanned by closed leaves of the foliation $\mathcal{F}_q$ and of $\hat{\mathcal{F}}_q$, and $I^+(\alpha) \equiv P^+ I(\alpha)$, $I^-(\alpha) \equiv P^- I(\alpha)/\mathcal{C}$, $I^\mathcal{C}(\alpha)\equiv P^\mathcal{C}I(\alpha)$.

\begin{lemma}
\label{lem:isotropics}
Let $q$ be a periodic quadratic differential, $\alpha = \sqrt{\pi_\kappa^* q}$, and let $\{a_1,\dots,a_k\}$ be a basis for $I(q)$ given by the homology classes of waistcurves of the periodic foliation given by $q$. Then
$$\mathrm{span}\, \{P^\pm \hat{a}_1,\dots,P^\pm \hat{a}_k\} = I^\pm(\alpha),$$
where $\hat{a}_i$ are lifts of $a_i$, i.e., they are the homology classes of the lifts of waistcurves which represent a basis of $I(q)$.
\end{lemma}
\begin{proof}
Let $\hat{a}_j$ be the homology class of the waistcurve of \emph{one} lift of the waistcurve $|a_j|$ of a cylinder defined by the foliation induced by $q$. Then $a_j = \sum_{i=1}^k t_i a_i$ for some $t\in\mathbb{Z}^k$ since $\{a_1,\dots,a_k\}$ is a basis of $I(q)$. Suppose then $\hat{a}_j = \sum_{i=1}^k t_i\hat{a}_i + e$, where, for $1\leq i \leq k$, $\hat{a}_i$ is the homology class of a lift of a representative of $a_i$ and let $e^\pm \equiv P^\pm e$.

We claim $e^+=0$. Otherwise $\sum_{i=1}^k t_i a_i + \pi_{\kappa*}(e^+) = a_j = \sum_{i=1}^k t_i a_i$, a contradiction since $\pi_{\kappa*}$ restricted to $ H_1^+(\hat{M};\mathbb{Q})$ is an isomorphism onto $H_1(M;\mathbb{Q})$. Thus,
$$\mathrm{span}\, \{P^+ \hat{a}_1,\dots,P^+ \hat{a}_k\} = I^+(\alpha)$$ 
and we have $\hat{a}_j = \sum_{i=1}^kt_i\hat{a}_i + e^-$ with $e^-\in H_1^-(\hat{M};\mathbb{Z})$. Let $m\in\mathbb{N}$ be unique positive integer so that $\frac{e^-}{m}\in H_1(\hat{M};\mathbb{Z})$ is primitive. Then
$$\frac{1}{m}\left(\hat{a}_j-\sum_{i=1}^k t_i \hat{a}_i\right) = \frac{e^-}{m} = \hat{e},$$
by \cite{meyersonRepresenting}, can be represented by a \emph{simple} closed curve $\gamma_{\hat{e}}$ which is not homologous to zero and which can be chosen so that $\gamma_{\hat{e}}\cap\sigma\circ\gamma_{\hat{e}} = \varnothing$. Let $\gamma_e\equiv\pi_\kappa\circ\gamma_{\hat{e}}$ be its image on $M$, which is a simple closed curve homologous to zero. This implies $e^-\in\mathcal{C}$ since elements of $\mathcal{C}$ are homology classes represented by lifts of curves on $M$ which are homologous to zero. Thus, modulo cycles in $\mathcal{C}$, $P^-\hat{a}_j = \sum_{i=1}^k t_i P^-\hat{a}_i$, and the result follows.
\end{proof}

\begin{lemma}
\label{lem:rank}
Let $q$ be a periodic quadratic differential. If $\dim I(q) = k$ then 
\begin{equation}
\label{eqn:rank}
\dim I^+ (\alpha) = k \geq \dim I^- (\alpha),
\end{equation}
equality holding if $q$ has at least two odd singularities.
\end{lemma}

\begin{proof}
If $\{a_q^1,\dots,a_q^k\}$ is a basis of $I(q)$, by Lemma \ref{lem:monodromy} we have $m(a_q^i) = 0$, thus every waistcurve $|a_q^i|$ from this set lifts to two different waistcurves $|a_i^\pm|$ and hence $\dim  \langle a_1^+,\dots,a_k^+,a_1^-,\dots,a_k^-\rangle \leq 2k.$
By changing basis through $P^\pm$,
\begin{equation}
\label{eqn:ineq}
\dim \mathrm{span}\, \{P^+a_1^+,\dots,P^+a_k^+\} + \dim \mathrm{span}\, \{P^-a_1^+,\dots,P^-a_k^+\}  \leq 2k.
\end{equation}
Since $\pi_{\kappa*}$ restricted to $ H_1^+(\hat{M};\mathbb{Q})$ is an isomorphism onto $H_1(M;\mathbb{Q})$, $k = \dim I(q) = \dim \mathrm{span}\, \{P^+a_1^+,\dots,P^+a_k^+\}$ which, combined with (\ref{eqn:ineq}) and Lemma \ref{lem:isotropics}, gives (\ref{eqn:rank}). 

To address the case of odd singularities, we first show that $a_i^+\neq \pm a_i^-$ for all $i\in\{1,\dots,k\}$. Suppose $a_i^+ =  \pm a_i^-$ holds for some $i$. Then $\hat{M}\backslash (|a_i^+|\cup |a_i^-|)$ is a disjoint union of punctured Riemann surfaces $S_1 \coprod S_2$, each of which maps to itself under $\sigma$ since $q$ has odd singularities and thus $\sigma$ has fixed points. This implies that $M\backslash |a_i|$ is disconnected, or $a_i = 0$, a contradiction. 
Thus the lift of each waistcurve satisfies $P^\pm a_i^+\neq 0$ for all $i$.

It remains to show that $\dim \langle P^- a_1^+,\dots, P^- a_k^+ \rangle = k$. Let $\{\bar{a}_1,b_1,\dots,\bar{a}_g,b_g\}$ be a completion of a basis of $I(q)$ to a symplectic basis of $H_1(M;\mathbb{Z})$ with $\bar{a}_i = a_q^i$ for $1\leq i \leq k$, $\bar{a}_i\cap b_j = \delta_i^j$, and $\bar{a}_i\cap\bar{a}_j = b_i\cap b_j = 0$ for all $i, j$. Suppose such basis is represented by simple closed curves $\gamma_{\bar{a}_i}$ and $\gamma_{b_i}$ with trivial holonomy, which can be assumed since there are at least two odd singularities. Then each such curve has two lifts $\gamma^\pm_{\bar{a}_i}$ and $\gamma^\pm_{b_i}$ with $[\gamma^\pm_{\bar{a}_i}] = a_i^\pm$ and $[\gamma^\pm_{b_i}] = b_i^\pm$. We can assign the $\pm$ labels such that $a_i^+ \cap b_i^+ = a_i^- \cap b_i^- \neq 0$ and all other intersections are zero. Indeed, starting with a symplectic basis such that $a_i \cap b_i = 1$ then there exist two closed curves $\gamma_{a_i}$ and $\gamma_{b_i}$ with trivial monodromy representing, respectively, $a_i$ and $b_i$, and intersecting only once on $M\backslash \Sigma_\mu$. The point of intersection has two lifts, which means there are two intersections on $\hat{M}\backslash \hat{\Sigma}$, from which the $\pm$ labels are assigned so that $a^+_i \cap b^+_i =a^-_i \cap b^-_i \not =0$ and all other intersections are zero. By repeating this procedure one obtains a basis for $H_1(\hat{M};\mathbb{Q})$ with the desired intersection properties. Since $\mathcal{C}$ is symplectically orthogonal to $I^+(\alpha)$ and $I^-(\alpha)$ we do not worry about the intersection with cycles in $\mathcal{C}$.

Suppose $\dim \langle P^- a_1^+,\dots, P^- a_k^+\rangle < k$. Without loss of generality we can assume there exists a $c\in\mathbb{Q}^{k-1}$ such that 
$$P^-a_1^+ = \frac{1}{2}(a_1^+ - \sigma_* a_1^+) = \sum_{i=2}^k c_{i-1} (a_i^+ - \sigma_* a_i^+).$$
Since $P^- b_1^+$ is the symplectic dual of $P^-a_1^+$,
$$0 \neq P^-b_1^+\cap P^-a_1^+ = \sum_{i=2}^k c_{i-1} P^-b_1^+\cap(a_i^+ - \sigma_* a_i^+) = 0,$$
a contradiction since the right hand side involves a sum of intersections which are all zero.
\end{proof}

Let $\mathcal{L}^{h,v}_\kappa$ be the set of quadratic differentials $q\in\mathcal{Q}_\kappa$ for which the foliation $\hat{\mathcal{F}}_q^{h,v}$ is Lagrangian. 

\begin{proposition}
\label{prop:density}
The set $\mathcal{L}_\kappa^{h,v}$ is dense in $\mathcal{Q}_\kappa$.
\end{proposition}

We remark that \cite[Lemma 4.4]{forni:deviation} proves this statement in the case of $q$ being the square of an abelian differential. Thus this proof follows closely the ideas of that proof, making slight modifications. We briefly review the idea for abelian differentials. One begins with a periodic foliation given by a holomorphic 1-form. Since these foliations are dense in the moduli space, the proof is completed by showing that given any periodic foliation, one can make an arbitrary small perturbation to this form to obtain a 1-form whose foliation is periodic and whose isotropic span has larger dimension than that of the unperturbed foliation. By making finitely many perturbations (no more than the genus of the surface) one obtains a Lagrangian foliation.

For a quadratic differential $q\in\mathcal{Q}_\kappa$  the idea is similar but one has to proceed carefully. Since local coordinates of $\mathcal{Q}_\kappa$ are given by periods in $H_-^1(\hat{S},\hat{\Sigma}_\kappa;\mathbb{R})$, we can only make perturbations of $\alpha = \sqrt{\pi_\kappa^* q}$ in the anti-invariant subspace of $H^1(\hat{S},\hat{\Sigma}_\kappa;\mathbb{R})$ by an anti-invariant holomoprhic 1-form. From here, by virtue of Lemma \ref{lem:rank}, we can proceed as in \cite{forni:deviation} when there are at least two odd singularities. When there are no odd singularities, the space $H^1_-(\hat{M};\mathbb{R})$ is too small to give enough perturbations to grow isotropically to a Lagrangian foliation, so we perturb our holomorphic 1-form with anti-invariant \emph{relative} cocycles, i.e., exact forms of the form $df$ which are non-zero elements of $H^1(\hat{S},\hat{\Sigma};\mathbb{R})$ and satisfy $\sigma^* df = - df$. We will show that perturbing with these exact forms we may continue growing-out until we get a Lagrangian foliation.

\begin{proof}

We will consider two different cases: quadratic differentials with and without odd singularities.
\begin{case}[Quadratic differentials with at least two odd singularities]
Since periodic quadratic differentials form a dense subset of $\mathcal{Q}_\kappa$, when $\mathcal{Q}_\kappa$ is a stratum of quadratic differentials with at least two odd singularities, we will show that there is a Lagrangian foliation arbitrarily close to a periodic one which is not Lagrangian. 
\end{case}

Suppose $q$ is a quadratic differential with at least two odd singularities such that its horizontal foliation is periodic and that for $\alpha = \sqrt{\pi_\kappa^* q}$ we have $g > \dim I^\pm(\alpha) = k$ (Lemma \ref{lem:rank}). Let $\{|a^1_q|,\dots,|a_q^t|\}$ be the waistcurves of the cylinders of the periodic foliation on $M$. Then 
$$\dot{\mathcal{M}}\equiv M\backslash(|a_q^1|\cup\cdots\cup |a_q^t|)$$
is a genus $(g-k)$ surface (possibly disconnected) with $2t$ paired punctures. On a component $M_c\subset\dot{\mathcal{M}}$ which is of positive genus, let $\gamma_c:[0,1]\rightarrow M_c$ be a smooth simple closed curve which represents a cycle which is not homologous to a linear combination of boundary cycles and has empty intersection with the singularity set $\Sigma_\kappa$. 

Denoting by $i:\dot{\mathcal{M}}\hookrightarrow M$ the inclusion map, then $\gamma\equiv i\circ \gamma_c:[0,1]\rightarrow M$, by construction, satisfies the following properties. If we define the non-zero homology class $h\equiv [\gamma]\in H_1(M;\mathbb{Z})$, then $h\not\in I(q)$, $h\cap b = 0$ for any $b\in I(q)$, $\gamma\cap |a_q^1| = \cdots = \gamma\cap |a_q^t| = \varnothing$ and $\gamma\cap\Sigma_\kappa = \varnothing$. Furthermore, we can assume $m(h)= 0$, since we can always modify $\gamma_c$ slightly to go around an odd singularity of $q$ in order to force $m(h) = 0$. Since each $\gamma$ has trivial monodromy, it has two lifts $\gamma^\pm$ to $\hat{M}$ with $[\gamma^-] = \sigma_* [\gamma^+]$. Let $h^\pm = [\gamma^\pm]\in H_1(\hat{M};\mathbb{Z})$, which by construction satisfies $h^\pm \cap b = 0$ for any $b\in I^\pm(\alpha)$ and $h^\pm\not\in I^\pm(\alpha)$. 

We claim $h^+\neq \pm h^-$. Suppose $h^+ =  \pm h^-$. Then $\hat{M}\backslash (\gamma^+\cup \gamma^-)$ is a disjoint union of punctured Riemann surfaces $S_1 \coprod S_2$, each of which maps to itself under $\sigma$ since $q$ has odd singularities and thus $\sigma$ has fixed points. This implies that $M\backslash \gamma$ is disconnected, or $h = 0$, a contradiction. For the two lifts $\gamma^\pm$ on $\hat{M}$ of the cycle $\gamma$, we have $\gamma^\pm\cap |\hat{a}_\alpha^1| = \cdots = \gamma^\pm\cap |\hat{a}_\alpha^{\hat{t}}| = \varnothing$ and $\gamma^\pm\cap\hat{\Sigma}_\kappa = \varnothing$.

Let $\mathcal{V}^\pm(\gamma^\pm)\subset\subset\mathcal{U}^\pm(\gamma^\pm)$ be sufficiently small open tubular neighborhoods of $\gamma^\pm$ in $\hat{M}$ such that
\begin{equation}
\label{eqn:nhoods}
\overline{\mathcal{U}^+(\gamma^+)}\cap\overline{\mathcal{U}^-(\gamma^-)}=\varnothing,\hspace{.3 in} \overline{\mathcal{U}^\pm(\gamma^\pm)}\cap(|\hat{a}^1_\alpha|\cup\cdots\cup |\hat{a}_\alpha^{\hat{t}}|) = \varnothing,\hspace{.3 in} \overline{\mathcal{U}^\pm(\gamma^\pm)}\subset\hat{M}\backslash\hat{\Sigma}_\kappa
\end{equation}
and $\mathcal{U}^-(\gamma^-) = \sigma(\mathcal{U}^+(\gamma^+))$, $\mathcal{V}^-(\gamma^-) = \sigma(\mathcal{V}^+(\gamma^+))$. Let $\mathcal{U}^\pm_\epsilon(\gamma^\pm)$, $\epsilon\in\{0,1\}$, be the two connected components of $\mathcal{U}^\pm(\gamma^\pm)\backslash\gamma^\pm$ and $\mathcal{V}^\pm_\epsilon(\gamma^\pm) = \mathcal{V}^\pm(\gamma^\pm)\cap\mathcal{U}^\pm_\epsilon(\gamma^\pm)$. Let $\phi^\pm:\mathcal{U}^\pm\rightarrow \mathbb{R}$ be a smooth function such that
$$\phi^\pm(x) = \left\{\begin{array}{ll}
0 & \mbox{ for } x\in \mathcal{U}^\pm_0(\gamma^\pm)\backslash\mathcal{V}_0^+(\gamma^+), \\
1 & \mbox{ for } x\in \mathcal{U}^\pm_1(\gamma^\pm) \end{array} \right. $$
and define the closed 1-forms
\begin{equation}
\label{eqn:forms}
\lambda^\pm = \left\{\begin{array}{ll}
0 & \mbox{on $\hat{M} \backslash \mathcal{U}^\pm(\gamma^\pm)$} \\
d\phi^\pm & \mbox{on $\mathcal{U}^\pm(\gamma^\pm)$} \end{array} \right. , \hspace{.5in} \eta^- = P^-\lambda^+.
\end{equation}
We claim that $0\neq[\eta^-]\in H^1(\hat{M};\mathbb{Q})$. Indeed, since $\lambda^+$ is dual to $h^+$ and $\sigma^*\lambda^+$ is dual to $h^-$, it follows from the fact that $h^+\neq h^-$.

The horizontal foliation given by $\alpha'_r = \alpha + r\eta^-$ for $r\in\mathbb{Q}$ sufficiently small is periodic and satisfies, by construction, the property that every waistcurve of $\mathcal{F}_\alpha$ is homologous to a waistcurve of $\mathcal{F}_{\alpha_r'}$ and therefore $I(\alpha)\subset I(\alpha'_r)$. This is a strict inclusion, since $P^{-1} \alpha'_r = P^{-1}\alpha + P^-h^+$ and by construction $h^\pm\not\in I(\alpha)$. Let $h^*\in H^-_1(\hat{M};\mathbb{Q})$ be such that $h^* \cap h^+ \neq 0$  and $h^* \cap a_\alpha^i = 0$ for all $i$. Since
$$\int_{h^*} \alpha'_r \neq 0,$$
by Lemma \ref{lem:intersection} we have $\dim I^-(\alpha_r') > \dim I^-(\alpha)$. Moreover, by Lemma \ref{lem:rank}, $\dim I^+(\alpha'_r) = \dim I^-(\alpha'_r) > \dim I^+(\alpha) = \dim I^-(\alpha)$, in other words, we have ``grown'' isotropically. 

After finitely many iterations of this perturbation procedure we obtain a form $\alpha^-$ with $I^+(\alpha^-)$ a Lagrangian subspace of $H_1^+(\hat{M};\mathbb{Q})$. Since $H_1^+(\hat{M};\mathbb{Q})$, $\hat{H}_1^-(\hat{M};\mathbb{Q})$ and $\mathcal{C}$ are symplectically orthogonal and each a symplectic subspace of $H_1(\hat{M};\mathbb{Q})$, by Lemma \ref{lem:rank}, $I^-(\alpha^-)$ is also a Lagrangian subspace of $\hat{H}_1^-(\hat{M};\mathbb{Q})$. As in \cite{forni:deviation}, one may continue with the perturbation procedure to obtain a Lagrangian subspace $I^\mathcal{C}$ of the symplectic subspace $\mathcal{C}$ by making similar perturbations in $\mathcal{C}^*$. Thus the case of a quadratic differential with at least two odd singularities is proved.

\begin{case}[Quadratic differentials with no odd singularities]
Suppose $q$ is a periodic quadratic differential with no odd singularities. In this case the only shortcoming is that the space $H^1_-(\hat{M};\mathbb{Q})$ is not big enough to provide enough perturbations to create a Lagrangian subspace in $H_1^+(\hat{M};\mathbb{Q})$. Specifically, since in this case $\dim H_1^-(\hat{M};\mathbb{Q}) = 2g - 2$, if we begin with a periodic quadratic differential with $\dim I(q) = k < g$ after $g-k -1$ iterations of the perturbative procedure described in the previous case we may get an isotropic subspace in $H_1^+(\hat{M};\mathbb{Q})$ of dimension $g-1$. At this point we are unable to perturb in $H^1_-(\hat{M};\mathbb{Q})$, so we perturb with elements of $H^1_-(\hat{M},\hat{\Sigma}_\kappa;\mathbb{Q})$ since it is this space which gives local coordinates to $\mathcal{Q}_\kappa$. As in the case of periodic quadratic differentials with odd singularities, it will be sufficient to show there is one with a Lagrangian foliation which is arbitrarily close. 
\end{case}

Suppose $q\in\mathcal{Q}_\kappa$ is a periodic quadratic differential on the genus $g$ surface $M$ in a stratum with no odd singularities and $k^- = \dim I^-(\alpha)<g-1$. Let $h\in H_1^-(\hat{M};\mathbb{Q})$ be a cycle such that $h\not\in I^-(\alpha)$ and $h\cap b = 0$ for all $b\in I^-(\alpha)$. Let $\bar{h}\in\ H_1^-(\hat{M};\mathbb{Z})$ be the unique (up to a sign) primitive integer multiple of $h$.

We can proceed to perturb $\alpha$ by the Poincar\'{e} dual to $\bar{h}$ (which by construction is an element of $H^1_-(\hat{M};\mathbb{Q})$) as in (\ref{eqn:nhoods}) and (\ref{eqn:forms}). In this case, we do not have to worry about making sure the perturbation is done by the dual of an element in $H_1^-(\hat{M};\mathbb{Q})$ since we have guaranteed this by construction in the preceding paragraph. Thus we obtain a new form $\alpha'$ with $\dim I^-(\alpha') > \dim I^-(\alpha)$ since the exact same arguments from Case 1 apply. After finitely many iterations of the previous perturbative procedure, each time with the Poincar\'{e} dual of an $\bar{h}$ as in the preceding paragraph, one can end up with a periodic foliation on $\hat{M}$ given by the Abelian differential $\alpha$ with $\dim  I^+(\alpha) = \dim  I^-(\alpha) = g-1$.  It could also happen that we obtain an Abelian differential with $g = \dim I^+(\alpha) > \dim I^-(\alpha) = g-1$ at which point the proposition would be proved for quadratic differentials with no odd singularities. In what follows, we treat the case $\dim I^\pm (\alpha) = g-1$.

Let $\{|a_1^+|,\dots,|a_{g-1}^+|,|a_1^-|,\dots,|a_{g-1}^-|\}$ be waistcurves of cylinders of the foliation given by $\alpha$ which represent a basis in homology for $I^+(\alpha)\oplus I^-(\alpha)$. Then 
$$\dot{\mathcal{M}} \equiv \hat{M}\backslash (|a_1^+| \cup \cdots \cup |a_{g-1}^+| \cup |a_1^-| \cup \cdots \cup |a_{g-1}^-|)$$
 is topologically a torus with $2g-2$ paired punctures coming from removing the waistcurves of cylinders. Let $p_1^+$ be a zero of $\alpha$ and $p^-_1 = \sigma(p^+_1)$.

Let $\gamma^+_1:[0,1]\rightarrow \hat{M}\backslash \mathcal{N}^1_\delta$, where $\mathcal{N}^1_\delta$ is a $\delta$ neighborhood of the punctures for some $\delta>0$, be a path on $\hat{M}$ such that 
\begin{equation}
\label{eqn:path}
\gamma^+_1(0) = p^-_1,\,\,\, \gamma^+_1(1) = p^+_1,\,\,\, \hat{\Sigma}_\kappa\cap \{\gamma_1^+(t)\}_{t\in(0,1)} = \varnothing,\,\,\, \mbox{and}\,\,\, 0\neq [\pi_\kappa \gamma^+_1]\in H_1(M;\mathbb{Z}). 
\end{equation}
Denote by $\gamma^-_1 = \sigma(\gamma^+_1)$ its image path satisfying $\gamma^\pm_1(\epsilon) = \gamma^\mp_1(1-\epsilon)$, $\epsilon \in \{ 0, 1\}$. Note that $0 \neq [\gamma^+_1\cup\gamma^-_1] \in H_1(\hat{M};\mathbb{Q})$ and $0 \neq P^-[\gamma^+_1] \in H_1(\hat{M},\hat{\Sigma};\mathbb{Q})$.

Let $\mathcal{U}^\pm_\varepsilon = B(p^\pm_1,\varepsilon)$ be two open $\varepsilon$-balls around $p^+_1$ and $p^-_1$ and $\mathcal{V}^\varepsilon_1$ a $\varepsilon$-tubular neighborhood around $\gamma^+_1\cup\gamma^-_1$. Let $f_1$ be a smooth function compactly supported in $\mathcal{V}^\varepsilon_1$ such that
\begin{equation}
\label{eqn:function}
f_1(x) = \left\{\begin{array}{ll}
0 & \mbox{on $\mathcal{U}^-_\varepsilon$} \\
1 & \mbox{on $\mathcal{U}^+_\varepsilon$} \end{array} \right.
\end{equation}
and $f_1^\pm \equiv P^\pm f_1$. 


Let $\alpha'_{r_1} = \alpha + r_1\cdot df_1^-$ for $r_1\in\mathbb{Q}$ sufficiently small. Since $f_1$ is constant inside $\mathcal{U}^\pm_\varepsilon$, $df_1^- = 0$ in a neighborhood of $p^\pm_1$, $\alpha'_{r_1}$ is still an Abelian differential with a periodic foliation. Moreover, since $\gamma^\pm_1$ is disjoint from the waistcurves $|a_\alpha^i|$ for $\varepsilon$ sufficiently small, the waistcurves $|a_\alpha^i|$ persist under the perturbation and are close and homologous to the waistcurves $|a_{\alpha'_{r_1}}^j|$ of the Abelian differential $\alpha'_{r_1}$. 

We claim not only that the foliation given by $\alpha'_{r_1}$ has more cylinders than the one given by $\alpha$, but that the waistcurve of at least one of these cylinders has non-zero intersection with $\gamma^+_1$. Since
\begin{equation}
\label{eqn:intersection'}
\int_{\gamma^+_1}\alpha'_{r_1} \neq 0,
\end{equation}
the claim follows from Lemma \ref{lem:intersection}. At this point either $\dim I(\alpha'_{r_1})>\dim I(\alpha)$ or $\dim I(\alpha'_{r_1})=\dim I(\alpha)$. If the former occurs, since $H_1^+(\hat{M};\mathbb{Q})$ and $H_1^-(\hat{M};\mathbb{Q})$ are symplectically orthogonal, $\dim H_1^-(\hat{M};\mathbb{Q}) = 2g-2$, and $\dim I^-(\alpha) = g-1$, this is equivalent to $\dim I^+(\alpha'_{r_1}) > \dim I^+(\alpha)$, and this completes the proof for quadratic differentials with no odd singularities.

Suppose $\dim I(\alpha'_{r_1}) = \dim I(\alpha)$. Let $c_{\alpha'_{r_1}}^*$ be a cylinder of the foliation given by $\alpha'_{r_1}$ such that $a_{\alpha'_{r_1}}^* \cap [\gamma^+_1] \neq 0$ in the sense of Lemma \ref{lem:intersection}. Clearly we have $a_{\alpha'_{r_1}}^*\cap a_{\alpha}^i = 0$ for any other waistcurve $a_\alpha^i$ of the foliation given by $\alpha$. 

Let $\mathcal{M}$ be a torus obtained by inserting $2g-2$ copies $\{D_i\}_{i=1}^{2g-2}$ of the two-disk to the punctures of $\dot{\mathcal{M}}$. Let $\theta_1$ be the \emph{closed} 1-form on $\mathcal{M}$ defined as
\begin{equation}
\label{eqn:closedform}
\theta_1 = \left\{\begin{array}{ll}
\alpha'_{r_1} & \mbox{on $\dot{\mathcal{M}}$} \\
\omega_i & \mbox{on $D_i$} \end{array} \right. ,
\end{equation}
where the $\omega_i$ are smooth forms outside finitely many singularities in the interior of each $D_i$ and are defined such that (\ref{eqn:closedform}) defines a smooth, closed form outside finitely many points. Then $\theta_1$ defines an orientable foliation on $\mathcal{M}$ which coincides with $\alpha'_{r_1}$ outside the inserted disks $D_i$. It follows from the Poincar\'{e}-Hopf index formula that if a simply connected, planar domain bounded by a periodic orbit of a vector field contains finitely many fixed points, the sum of the indices at every fixed point in the interior is equal to 1. In other words, denoting by $\iota_p(\theta)$ the index of the vector field (foliation) given by $\theta$ at the singularity $p$, we have
$$\sum_{p\in\mathrm{int}(D_i)}\iota_p(\theta) = 1$$
for any $i$ since $D_i$ is a simply connected, bounded planar domain. If $\dim I(\alpha'_{r_1}) = \dim I(\alpha)$ both the waistcurve $|a^*_{\alpha'_{r_1}}|$ and its image $\sigma |a^*_{\alpha'_{r_1}}|$ each bound a simply connected domain on $\mathcal{M}$. By (\ref{eqn:intersection'}), $p^+_1$ is contained in the interior of one of the two domains $B^+_1$ and $p^-_1$ in the other $B^-_1$. 
We claim that this finishes the proof for all differentials $q\in \mathcal{Q}_\kappa$ for $\kappa = \{4g-4\}$ for any $g>1$. Indeed, since $p^\pm_1$ were the only singularities of $\alpha$ and each was of negative index, by the Poincar\'{e}-Hopf index theorem,
\begin{equation}
\label{eqn:index}
0 = \chi(\mathcal{M}) = \sum_{p\in B^\pm_1} \iota_p(\theta_1) + \sum_{p\in (\mathcal{M}\backslash B^\pm_1)} \iota_p(\theta_1) = 2 + \sum_{p\in (\mathcal{M}\backslash B^\pm_1)} \iota_p(\theta_1) \geq 2,
\end{equation}
a contradiction. Thus neither $|a^*_{\alpha'_{r_1}}|$ or its image $\sigma |a^*_{\alpha'_{r_1}}|$ bound a simply connected domain, i.e.,  $\dim I(\alpha'_{r_1}) > \dim I (\alpha)$ and the proof is concluded in this case. 

After finitely many iterations of the above argument we can reach the same contradiction for any quadratic differential with no odd singularities. In fact, if $q\in Q_\kappa$ with $\kappa = \{n_1,\dots,n_\tau\}$ has no odd singularities, after no more than $\tau$ iterations, we reach the same contradiction. We show the argument for $\kappa = \{n_1,n_2\}$ with $n_1$, $n_2$ even and $n_1+n_2 = 4g-4$ for some $g>1$. For $\tau>2$, the argument is the same.

If after one iteration we do not reach a contradiction, we pick two other singularities $p^+_2$ and $p^-_2 = \sigma(p^+_2)$ of $\alpha'_{r_1}$ which are not in the interior of $B^\pm_1$ (if there are no such singularities, we reach the same contradiction through (\ref{eqn:index})). Define a path $\gamma_2^+:[0,1]\rightarrow \mathcal{M}\backslash \mathcal{N}^2_\delta$ as in (\ref{eqn:path}) where $\mathcal{N}^2_\delta$ is a $\delta$ neighborhood of the set $\{D_i\}_{i=1}^{2g-2}\cup B^\pm_1 \cup \mathcal{V}^\varepsilon_1$ for $\delta$ small enough. Let $f_2$ and $f^\pm_2$ be defined as in (\ref{eqn:function}) for $p^\pm_2$ and let $\theta_2 = \theta_1 + r_2\cdot df^-_2$ for a small enough $r_2\in \mathbb{Q}$. Then
$$\int_{\gamma^+_2}\theta_2 \neq 0$$
which, by Lemma \ref{lem:intersection}, implies there is a new cylinder given by the foliation which intersects $\gamma^+_2$. Note that we obtain the same form $\theta_2$ if we substitute the form $\alpha'_{r_2} = \alpha + r_1\cdot df_1^- + r_2\cdot df_2^-$ for $\alpha'_{r_1}$ in (\ref{eqn:closedform}), thus the new waistcurve given by $\theta_2$ also corresponds to a new waistcurve on $\hat{M}$ given by $\alpha'_{r_2}$.

If the waistcurve $|a_{\theta_2}|$ of this new cylinder represents a cycle which is homologous to zero, that is, if $\dim I(\alpha'_{r_2}) = \dim I(\alpha)$, then $|a_{\theta_2}|$ and its image $\sigma|a_{\theta_2}|$ bound simply connected domains $B^+_2$ and $B^-_2$ containing $p^+_2$ and $p^-_2$, respectively, on $\mathcal{M}$.  As in (\ref{eqn:index}),
\begin{eqnarray*}
0 = \chi(\mathcal{M}) = \sum_{i\in\{1,2\}}\sum_{p\in B^\pm_i} \iota_p(\theta_2) &+& \sum_{p\in (\mathcal{M}\backslash (B^\pm_1\cup B^\pm_2))} \iota_p(\theta_2) \\
= 2 &+& \sum_{p\in (\mathcal{M}\backslash (B^\pm_1\cup B^\pm_2))} \iota_p(\theta_2) \geq 2,
\end{eqnarray*}
since the only singularities of $\theta_2$ of negative index were in $B^\pm_1$ and $B^\pm_2$. Thus we get the same contradiction as in (\ref{eqn:index}). For an arbitrary stratum with no odd singularities, we can continue the same perturbation procedure with different anti-invariant relative cocycles which are dual to relative cycles connecting paired zeros at every step. After finitely many perturbations (no more than $\tau$) each zero of $\alpha$ (singularity of negative index) is contained in a simply connected domain of the foliation, which leads to a contradiction through the Poincar\'{e}-Hopf index formula. Thus, at some point of the perturbative procedure with relative, anti-invariant cycles, we obtain $\dim I^+(\alpha'_{r_i}) > \dim I^+(\alpha)$ and thus a Lagrangian foliation on $\hat{M}$.
\end{proof}

Finally we can prove the main theorem of this paper.
\begin{theorem}
\label{thm:main}
The Kontsevich-Zorich cocycle is non-uniformly hyperbolic $\hat{\mu}_\kappa$-almost everywhere on $\mathcal{H}_{\hat{\kappa}}$, where $\hat{\mu}_\kappa$ is the measure (\ref{eqn:measure}) supported on abelian differentials which come from non-orientable quadratic differentials through the double cover construction. The Lyapunov exponents satisfy
\begin{equation}
\label{eqn:spectrum}
1=\lambda_1>\lambda_2\geq\cdots\geq\lambda_g> 0 > \lambda_{g+1} = -\lambda_g \geq \cdots\geq \lambda_{2g-1} = -\lambda_2 > \lambda_{2g} = -1.
\end{equation}
\end{theorem}

Since the Kontsevich-Zorich cocycle defines two cocycles on the bundle over $i_\kappa(\mathcal{Q}_\kappa)\subset\mathcal{H}_{\hat{\kappa}}$, namely, the restriction of the cocycle to the invariant and anti-invariant sub-bundles (which are each invariant under the action of the cocycle), Theorem \ref{thm:main} implies we can express the Lyapunov exponents of the Kontsevich-Zorich cocycle of the invariant and anti-invariant sub-bundles as
$$1>\lambda_1^+ \geq \lambda_2^+ \geq \dots \geq \lambda_g^+ > 0 > -\lambda_g^+ = \lambda_{g+1}^+\geq\dots\geq \lambda_{2g}^+>-1$$
and
$$1=\lambda_1^- > \lambda_2^- \geq \dots \geq \lambda_{g+n-1}^- > 0 > -\lambda_{g+n-1}^- = \lambda_{g+n}^-\geq\dots> \lambda_{2g+2n-2}^-=-1$$
since, by the remark following Theorem \ref{thm:criterion}, the sub-bundles corresponding to the simple, extreme exponents are respectively generated by $[\mathrm{Re}\,\sqrt{\pi_\kappa^*q}]\cdot\mathbb{R}$ and $[\mathrm{Im}\,\sqrt{\pi_\kappa^*q}]\cdot\mathbb{R}$.

The question about the simplicity of the Lyapunov spectrum remains open. Examples of non-simple spectrum (in fact, degenerate spectrum, i.e., $\lambda_i=0$ for all $i\neq 1$) for other measures usually involve a certain set of symmetries (see \cite{FMZ} for a thorough discussion and examples) which are not present for Lebesgue almost all non-orientable quadratic differentials. The involution $\sigma$ splits the cocycle into two symplectic cocycles and it would be very surprising to find strong enough symmetries from the involution which would imply non-simplicity of the spectrum (\ref{eqn:spectrum}). Numerical experiments indeed show strong evidence for a simple spectrum. Thus we conjecture that for $\hat{\mu}_\kappa$-almost all quadratic differentials, the Kontsevich-Zorich cocycle has simple spectrum. We have approximated numerically the values of the exponents for several strata, which we summarize in the appendix.

We remark that Proposition \ref{prop:density} is stronger than needed to prove the result, as Forni's criterion needs \emph{one} Lagrangian differential in the support of the measure. It is thus possible to prove Theorem \ref{thm:main} through other methods by showing there is at least one Lagrangian differential in the support of the canonical measure such that not only $\mathcal{F}_q$ is Lagrangian on $M$, but also that $\hat{\mathcal{F}}_q$ is Lagrangian on $\hat{M}$. It seems that the tools from generalized permutations (see for example \cite{BoissyLanneau}) could be used to obtain such results, although we believe in such case it the hardest task would be to obtain a Lagrangian subspace $I^\mathcal{C}$ in the symplectic subspace $\mathcal{C}$ in the case of many odd singularities. In the same case, showing that $I^\pm$ are Lagrangian would not be a difficult task since, by Lemma \ref{lem:rank}, it suffices to obtain a Lagrangian foliation on $M$. The case of quadratic differentials with no odd singularities would most likely also have to be treated as a special case as well. We would be very interested to see whether Theorem \ref{thm:main} can be proved in such way (the tools and results of \cite{fickenscher} look particularly promising for this task).

\begin{proof}[Proof of Theorem \ref{thm:main}]
Since the measure (\ref{eqn:measure}) is the push-forward of a canonical measure which is locally equivalent to Lebesgue by the period map, it is easy to see that it has local product structure. By Proposition \ref{prop:density}, quadratic differentials $q$ such that $\hat{\mathcal{F}}_q$ is Lagrangian are dense in every stratum of $\mathcal{Q}_\kappa$ and thus the measure $\hat{\mu}_\kappa$ on $\mathcal{H}_{\hat{\kappa}}$ is cuspidal Lagrangian. The theorem then follows from Forni's criterion, Theorem \ref{thm:criterion}.
\end{proof}

\section{Deviation Phenomena}
\label{sec:deviations}
Let $M$ be a smooth, closed manifold and $X$ a smooth vector field on $M$ which generates a flow $\varphi_t$. For a point $p\in M$, let $c_T(p)\in H_1(M;\mathbb{R})$ be the cycle represented by closing the segment $\varphi_T(p)$ by a shortest path joining $\varphi_T(p)$ to $p$. For an ergodic measure $\mu$, invariant under $X$, and a point $p$ the support of $\mu$, the \emph{Schwartzman asymptotic cycle} \cite{schwartzman:cycle} is defined as
$$c_\mu^*\equiv \lim_{T\rightarrow \infty}\frac{c_T(p)}{T} \in H_1(M;\mathbb{R}).$$
The cycle $c_\mu^*$ is a sort of topological invariant of the flow $X$ with respect to the measure $\mu$ which can be regarded as a generalization of a rotation number since it coincides with the usual notion of rotation number for a minimal flow on a torus.

In the case when $M$ is an closed, orientable surface of genus $g>1$ endowed with a flat metric outside finitely many singular points and $X$ generates a (uniquely ergodic) translation flow (in other words, straight-line flow on a translation surface) on $M$, Zorich \cite{zorich-leaves} observed the following unexpected deviation phenomena through computational experiments. There are $g$ numbers $1=\lambda_1>\dots>\lambda_g>0$ and a filtration of subspaces
$$\langle c^*\rangle = F_1 \subset \cdots\subset F_g\subset H_1(\hat{M};\mathbb{R})$$
with $\dim F_i/F_{i-1} = 1$ such that, for $\phi\in\mathrm{Ann}(F_i)$ but $\phi\not\in\mathrm{Ann}(F_{i+1})$,
\begin{equation}
\label{eqn:deviationHomology}
\limsup_{T\rightarrow \infty}\frac{\log \|\langle \phi, c_T\rangle\|}{\log T} = \lambda_{i+1}
\end{equation}
Cycles which generate the subspaces $F_i$ are called \emph{Zorich cycles}. It was also proved that the numbers $\lambda_i$ actually coincide with the Lyapunov exponents (\ref{eqn:exponents}) of the Kontsevich-Zorich cocycle. In fact, he proved the following conditional statement.
\begin{theorem}[\cite{zorich-leaves}]
\label{thm:conditional}
Suppose the Kontsevich-Zorich cocycle is non-uniformly hyperbolic, i.e., $\lambda_1 \geq \dots\geq \lambda_g > 0$, and let $\lambda_1'>\dots>\lambda_s'>0$ be the different Lyapunov exponents. Then there exists a filtration of subspaces in $H_1(M;\mathbb{R})$
$$\langle c^*\rangle = F_1 \subseteq \cdots\subseteq F_s\subset H_1(\hat{M};\mathbb{R})$$
with $\dim F_i/F_{i-1} = \mbox{ multiplicity of }\lambda_i'$, $\dim F_s = g$ such that (\ref{eqn:deviationHomology}) holds. Moreover $[c_T(p)]$ remains within bounded distance of $F_s$ for almost every point $p$.
\end{theorem}
Based on the computer experiments, it was conjectured by Kontsevich and Zorich that for the canonical measure on the moduli space of orientable quadratic differentials, the Kontsevich-Zorich cocycle is non-uniformly hyperbolic and has a simple spectrum. This became known as the \emph{Kontsevich-Zorich conjecture} \cite{kontsevich:hodge}. It was also conjectured that similar deviations should hold for ergodic averages of smooth functions. Specifically, it was conjectured that for a smooth function $f$ and large $T$,
\begin{equation}
\label{eqn:deviationAverages}
\left|\int_0^T f\circ\varphi_t(p)\, dt \right| \approx T^{\lambda_{i+1}}
\end{equation}
for almost every $p$ on a codimension $i$ subspace in some space of functions.

The non-uniform hyperbolicity of the Kontsevich-Zorich cocycle was first proved in \cite{forni:deviation}. There it was proved that the deviation of ergodic averages is in fact described  by the exponents of the Kontsevich-Zorich cocycle and that $\lambda_g>0$, but the simplicity of the spectrum was not proved for surfaces of genus greater than two. The full conjecture, that is, that the spectrum of the cocycle is simple and that $\lambda_g > 0$, was proved by Avila and Viana \cite{AvilaViana}. We now recall the precise results on deviations of ergodic averages from \cite[\S6-\S9]{forni:deviation}.

Let $X_\alpha$ be a vector field on a surface $M$ of genus $g$ which is tangent to the horizontal foliation of an abelian differential $\alpha$. Let $\mathcal{I}_{X_\alpha}^1(M)$ denote the vector space of $X_\alpha$-invariant distributions (in the sense of Schwartz), i.e., distributional solutions $\mathcal{D}\in H^{-1}(M)$ of the equation $X_\alpha \mathcal{D} = 0$, where $H^{-1}(M)$ is the dual space of the Sobolev space $H^1(M)$.
\begin{theorem}[\cite{forni:deviation}]
\label{thm:forni}
For Lebesgue-almost all abelian differentials $\alpha$ the space $\mathcal{I}_{X_\alpha}^1(M)$ has dimension $g$ and there exists a splitting
$$\mathcal{I}^1_{X_\alpha}(M) = \mathcal{I}^1_{X_\alpha}(\lambda_1')\oplus\dots\oplus\mathcal{I}^1_{X_\alpha}(\lambda_s')$$
where $\dim \mathcal{I}^1_{X_\alpha}(\lambda_i') = \mbox{multiplicity of }\lambda_i'$ for the $i^{th}$ distinct Lyapunov exponent of Kontsevich-Zorich cocycle. Denoting by $\varphi_t$ the flow of $X_\alpha$, for any function $f\in H^1(M)$ such that
$$\mathcal{D}f=0\hspace{.4 in}\mbox{ for all }\hspace{.4 in}\mathcal{D}\in\mathcal{I}^1_{X_\alpha}(\lambda_1')\oplus\dots\oplus\mathcal{I}^1_{X_\alpha}(\lambda_i'),$$
and if there exists a $\mathcal{D}_{i+1}\in\mathcal{I}^1_{X_\alpha}(\lambda_{i+i}')\backslash\{0\}$ such that $\mathcal{D}_{i+1}f\neq 0$,
then, if $0<i<s$, for almost every $p\in M$, 
$$\limsup_{T\rightarrow\infty}\frac{\log|\int_0^Tf\circ\varphi_t(p)\,dt|}{\log T}= \lambda_{i+1}'.$$
If $\mathcal{D}f=0$ for all $\mathcal{D}\in\mathcal{I}^1_{X_\alpha}$, then for any $p$ not contained in a singular leaf, 
$$\limsup_{T\rightarrow\infty}\frac{\log|\int_0^Tf\circ\varphi_t(p)\,dt|}{\log T}=0.$$
\end{theorem}

A \emph{basic current} $C$ for $\mathcal{F}$ is a current (in the sense of de Rham) of dimension and degree equal to one such that for all vector fields $X$ tangent to $\mathcal{F}$ we have
$$i_X C = \mathcal{L}_X C = 0.$$
Let $\mathcal{B}^s_q$ be the space of currents for $\mathcal{F}_q^h$ of order $s$. It was proved in \cite{forni:deviation} that the space $\mathcal{I}^1_{X_\alpha}$ is in bijection with the subspace $\mathcal{B}^1_{q,+}\subset\mathcal{B}^1_q$ of \emph{closed} currents which are not exact. In fact, $C\in\mathcal{B}^1_{q,+}$ if and only if $C\wedge [\mathrm{Im}\,\alpha]\in\mathcal{I}^1_{X_\alpha}$. There is an analogous splitting of the space $\mathcal{B}^1_{q,+}$:
$$\mathcal{B}^1_{q,+} = \mathcal{B}^1_{q,+}(\lambda_1')\oplus\dots\oplus\mathcal{B}^1_{q,+}(\lambda_s')$$
with respect to the Lyapunov spectrum of the Kontsevich-Zorich cocycle. Let 
$$\Pi^i_q:\mathcal{B}^1_{q}\longrightarrow\mathcal{B}^1_{q,+}(\lambda_i')$$
be the projection to the $i^{th}$ summand of the splitting. The invariant distributions which generate each $\mathcal{I}^1_{X_\alpha}(\lambda_i')$ are constructed from the asymptotic currents as follows. There is a sequence of times $T_k\rightarrow \infty$ such that
\begin{equation}
\label{eqn:asCurrent}
\mathcal{D}_i\equiv \lim_{k\rightarrow\infty} \frac{\Pi^i_q\, \ell_{T_k}\wedge[\mathrm{Im}\,\alpha]}{| \Pi^i_q\,\ell_{T_k}|_{-1}} = \lim_{k\rightarrow\infty} \frac{\Pi^i_q\, \ell_{T_k}}{| \Pi^i_q\,\ell_{T_k}|_{-1}}\wedge[\mathrm{Im}\,\alpha] = C_i\wedge[\mathrm{Im}\,\alpha]\in\mathcal{I}^1_{X_\alpha}(\lambda_i'),
\end{equation}
are the invariant distributions, where $\ell_T$ is the current defined by a segment of a leaf of $\mathcal{F}^h_q$ (a chain) of length $T$. Furthermore,
\begin{equation}
\label{eqn:currNorm}
\limsup_{T\rightarrow\infty}\frac{\log|\Pi^i_q\,\ell_{T}|_{-1}}{\log T}= \lambda_i'.
\end{equation}
Thus, the basic currents $C_i$ in (\ref{eqn:asCurrent}) are the Zorich cycles which generate the subspaces $F_i$ in Theorem \ref{thm:conditional}. In fact, there is a representation theorem of Zorich cycles which states that all Zorich cycles are represented by basic currents of order 1 \cite[Theorem 8.3]{forni:deviation}.

Any element of the spectrum of the Kontsevich-Zorich cocycle for the canonical measure in the moduli space of abelian differentials describes deviations of both homology cycles as well as that of ergodic averages. For the case of non-orientable quadratic differentials, it is surprisingly not the same.

\subsection{Deviations in homology for quadratic differentials}

Let $q\in \mathcal{Q}_\kappa$ be a quadratic differential on $M$ which is an Oseledets-regular point with respect to the measure (\ref{eqn:measure}) for the Kontsevich-Zorich cocycle. Let $\hat{M}$ the orienting double cover and $\alpha = \sqrt{\pi_\kappa^* q}$. For a point $p\in M$ on a minimal leaf $\ell$ of $\mathcal{F}_q$ and picking a local direction, we can follow a segment of length $T$, $\ell_T$, of the leaf $\ell$ in such direction. Let $c_T\in H_1(M;\mathbb{R})$ be the cycle obtained by closing the chain $\ell_T$ by a short path. 

For a point $\hat{p}\in\pi_\kappa^{-1}(p)$, following a leaf $\hat{\ell}_T$ of length $T$ of the foliation $\hat{\mathcal{F}}_q$ such that $\pi_\kappa \hat{\ell}_T = \ell_T$, let $\hat{c}_T\in H_1(\hat{M};\mathbb{R})$ be the cycle obtained by closing the chain $\hat{\ell}_T$ by a short path. Then
$$\hat{c}_q^*\equiv \lim_{T\rightarrow \infty}\frac{\hat{c}_T}{T} \in H_1^-(\hat{M};\mathbb{R})$$
is the Schwartzman asymptotic cycle. It is anti-invariant with respect to $\sigma_*$ since it can be shown to be the Poincar\'{e} dual of the cohomology class defining the foliation, in this case either $[\mathrm{Re}(\alpha)]$ or $[\mathrm{Im}(\alpha)]$. By construction, $\pi_{\kappa*} \hat{c}_q^* = c_q^*$. By Theorem \ref{thm:main}, the Kontsevich-Zorich cocycle is non-uniformly hyperbolic with respect to the measure $\hat{\mu}_\kappa$ supported on $i_\kappa(\mathcal{Q}_\kappa)\subset\mathcal{H}_{\hat{\kappa}}$. Let $1=\lambda_1^-> \lambda_2^-\geq \cdots \geq \lambda_{2g+n-1}^->0$ and $\lambda_1^+\geq\cdots\geq\lambda_g^+>0$ be the positive Lyapunov exponents of the restriction of the cocycle to the anti-invariant and invariant sub-bundles, respectively. By Theorem \ref{thm:conditional}, for large $T$,
$$\hat{c}_T \approx \underbrace{\hat{c}_q^* T + c_2^- T^{\lambda_2^-}+\cdots}_{\mbox{coming from }H_1^-(\hat{M};\mathbb{R})}+ \underbrace{ c_1^+ T^{\lambda_1^+} + c_2^+ T^{\lambda_2^+}+ \cdots}_{\mbox{coming from }H_1^+(\hat{M};\mathbb{R}) }.$$
Since $\pi_{\kappa*}\hat{c}_T = c_T$ and $\ker\pi_{\kappa*} = H_1^-(\hat{M};\mathbb{R})$,
\begin{equation}
\label{eqn:cycle}
c_T \approx \pi_{\kappa*}(c_1^+) T^{\lambda_1^+} + \pi_{\kappa*}(c_2^+) T^{\lambda_2^+}+ \cdots.
\end{equation}
If we define the Schwartzman asymptotic cycle for the non-orientable foliations on $M$ as
$$c_q^*\equiv \lim_{T\rightarrow \infty}\frac{c_T}{T} \in H_1(M;\mathbb{R}),$$
then, by (\ref{eqn:cycle}), it is well-defined and equal to zero. Thus the deviation of homology classes is sublinear and described completely by invariant behavior. The result is summarized in the following theorem.

\begin{theorem}[Deviations in homology for a typical leaf of a quadratic differential]
\label{thm:deviationHomology}
For Lebesgue-almost all quadratic differentials $q\in \mathcal{Q}_g$ on $M$, there exists a filtration of subspaces 
$$F_1\subset \cdots \subset F_s \subset H_1(M;\mathbb{R})$$
with $\dim F_i/F_{i-1} = \mbox{multiplicity of }\lambda^+_i$ and $F_s$ a Lagrangian subspace, such that, for $\phi\in\mathrm{Ann}(F_i)$ but $\phi\not\in\mathrm{Ann}(F_{i+1})$,
$$\limsup_{T\rightarrow \infty}\frac{\log \| \langle \phi, c_T \rangle \|}{\log T} = \lambda_{i+1}^+$$
where $c_T$ is obtained by closing a non-singular leaf $\ell_T$ of length $T$ by a short segment and $\lambda_1^+>\cdots>\lambda_s^+>0$ are the distinct Lyapunov exponents of the Kontsevich-Zorich cocycle with respect to the measure coming from quadratic differentials, restricted to the invariant sub-bundle $H^1_+(\hat{M};\mathbb{R})$.
\end{theorem}

\subsection{Deviation of ergodic averages for quadratic differentials}
Let $q\in \mathcal{Q}_\kappa$ be a quadratic differential on $M$ which is an Oseledets-regular point with respect to the measure (\ref{eqn:measure}) for the Kontsevich-Zorich cocycle. Let $\hat{M}$ be the orienting double cover and $\alpha = \sqrt{\pi_\kappa^* q}$. For a point $p\in M$ on a minimal leaf of $\mathcal{F}_q$, let $\varphi_t(p)$ be the ``flow'' obtained by integrating the distribution defining the horizontal foliation in a chosen direction and starting at $p$. As such, $\bigcup_{t=0}^T\varphi_t(p)$ is a segment $\ell_T$ of a leaf of the horizontal foliation $\mathcal{F}_q$ of length $T$ with an endpoint $p$. Then for a smooth function $f$,
\begin{equation}
\label{eqn:flow}
\int_0^T f\circ \varphi_s(p)\, ds
\end{equation}
is well defined. Let $\hat{f} = \pi_\kappa^* f$ be a smooth function on $\hat{M}$. Then
$$\int_0^T f\circ \varphi_s(p)\,ds = \int_0^T \hat{f}\circ\hat{\varphi}_s(\hat{p})\,ds,$$
for the flow $\hat{\varphi}_t(p)$ defined by the orientable horizontal foliation $\hat{\mathcal{F}}_q$ for a point $\hat{p}\in\pi_\kappa^{-1}(p)$. Moreover,
\begin{equation}
\label{eqn:averages}
\int_0^T f\circ \varphi_s(p)\,ds = \int_0^T \hat{f}\circ\hat{\varphi}_s(\hat{p})\,ds = \langle \ell_T, \hat{f}\cdot \mathrm{Im}\,\alpha\rangle = \int_{\ell_T}\hat{f}\cdot\mathrm{Im}\,\alpha.
\end{equation}
For the space of invariant distributions $\mathcal{I}^1_q(\hat{M})$, let $\mathcal{I}^\pm_q \equiv P^\pm\mathcal{I}^1_q(\hat{M})$. 
By \cite{forni:deviation}, there is a splitting of the closed, non-exact basic currents of order one 
$$\mathcal{B}^1_{q}=\mathcal{B}^+_q\oplus\mathcal{B}^-_q=\mathcal{B}^+_q(\lambda^+_1)\oplus\dots\oplus\mathcal{B}^+_q(\lambda^+_{s^+})\oplus\mathcal{B}^-_q(\lambda^-_1)\oplus\dots\oplus\mathcal{B}^-_q(\lambda^-_{s^-})$$
into the components corresponding to the Lyapunov exponents coming from the restriction of the cocycle to the invariant and anti-invariant sub-bundles, respectively. Let $\Pi^i_\pm:\mathcal{B}^1_q\longrightarrow\mathcal{B}^\pm_q(\lambda^\pm_i)$. For an invariant distribution $\mathcal{D} = C\wedge\mathrm{Im}\,[\alpha]$, since $[\alpha]\in H^1_-(\hat{M},\hat{\Sigma}_\kappa;\mathbb{R})$, $\mathcal{D}\in\mathcal{I}^\pm_q$ if and only if $C\in\mathcal{B}^\mp_q$.

If $\mathcal{D}\in\mathcal{I}^-_q$, $\mathcal{D}(\hat{f}) = 0$ for $\hat{f}=\pi_\kappa^* f$. Then, by (\ref{eqn:asCurrent}), (\ref{eqn:currNorm}), and (\ref{eqn:averages}), for large $T$,
\begin{eqnarray*}
\int_0^T f\circ\varphi_s(p)\,ds &\approx& \sum_{i=1}^{s^-}\langle \Pi^i_-\,\ell_T,\hat{f}\cdot\mathrm{Im}\,\alpha\rangle\cdot T^{\lambda^-_i} + \sum_{i=1}^{s^+}\langle \Pi^i_+\,\ell_T,\hat{f}\cdot\mathrm{Im}\,\alpha\rangle\cdot T^{\lambda^+_i} \\
&=& \sum_{i=1}^{s^-}\langle \Pi^i_-\,\ell_T,\hat{f}\cdot\mathrm{Im}\,\alpha\rangle\cdot T^{\lambda^-_i},
\end{eqnarray*}
and thus the deviation of ergodic averages are described by anti-invariant behavior. If $H^1(M)$ denotes the standard Sobolev space of functions on $M$, then it is clear to see that $\pi_{\kappa}^* H^1(M)\subset H^1(\hat{M})$.
The results of \cite{forni:deviation} and Theorem \ref{thm:main} imply the following.
\begin{theorem}[Deviations of ergodic averages for quadratic differentials]
\label{thm:deviationAverages}
For Lebesgue-almost all non-orientable differentials $q$ on a genus $g$ surface $M$ there is a space $\mathcal{I}_q^1(M)$ of dimension $2g+2n-2$ of distributions defined as the push-forward of the space of invariant distributions $\mathcal{I}_q^+$ on $\hat{M}$ which splits as
$$\mathcal{I}^1_q(M) = \pi_{\kappa*}\mathcal{I}^+_q(\lambda_1')\oplus\dots\oplus\pi_{\kappa*}\mathcal{I}^+_q(\lambda_{s^-}')$$
where $\dim \mathcal{I}^+_q(\lambda_i') = \mbox{multiplicity of }\lambda_i'$ for the $i^{th}$ distinct Lyapunov exponent of Kontsevich-Zorich cocycle restricted to the anti-invariant sub-bundle. Denoting by $\varphi_t$ the local flow of $\mathcal{F}^h_q$ as in (\ref{eqn:flow}), for any function $f\in H^1(M)$ such that
$$\mathcal{D}f=0\hspace{.4 in}\mbox{ for all }\hspace{.4 in}\mathcal{D}\in\pi_{\kappa*}\mathcal{I}^+_q(\lambda_1')\oplus\dots\oplus\pi_{\kappa*}\mathcal{I}^+_q(\lambda_i'),$$
and if there exists a $\mathcal{D}_{i+1}\in\pi_{\kappa*}\mathcal{I}^+_q(\lambda_{i+i}')\backslash\{0\}$ such that $\mathcal{D}_{i+1}f\neq 0$,
then, if $0<i<s^-$, for almost every $p\in M$,
$$\limsup_{T\rightarrow\infty}\frac{\log|\int_0^Tf\circ\varphi_t(p)\,dt|}{\log T}= \lambda_{i+1}'.$$
If $\mathcal{D}f=0$ for all $\mathcal{D}\in\pi_{\kappa*}\mathcal{I}^+_q$, then for any $p$ not contained in a singular leaf, 
$$\limsup_{T\rightarrow\infty}\frac{\log|\int_0^Tf\circ\varphi_t(p)\,dt|}{\log T}=0.$$
\end{theorem}
It is a consequence of a result of Masur and Smillie \cite{Masur-Smillie} that the anti-invariant sub-bundle can be arbitrarily large for a fixed genus $g$ surface. Consequently, by the above theorem, there are non-orientable foliations for which the space of invariant distributions $\mathcal{I}_q^1(M)$ can have arbitrarily large dimension and the deviation of ergodic averages are described by arbitrarily many parameters.

By (\ref{eqn:commute}), the Kontsevich-Zorich cocycle over $\mathcal{Q}_\kappa$ describes only the Lyapunov exponents of the invariant sub-bundle over $i_\kappa(\mathcal{Q}_\kappa)\subset \mathcal{H}_{\hat{\kappa}}$. Thus, by the above theorem, there seems to be no \emph{a-priori} reason for the Lyapunov exponents of the cocycle over $\mathcal{Q}_\kappa$ to describe the deviation of averages of functions along leaves of the foliation: only if there is repetition of exponents across the invariant and anti-invariant sub-bundles does the cocycle over $\mathcal{Q}_\kappa$ describe the deviation behavior of ergodic integrals.

\appendix
\section{Approximating the Lyapunov exponents numerically}

The Kontsevich-Zorich cocycle is a continuous-time version of a discrete, matrix-valued cocycle, the Rauzy-Veech-Zorich cocycle. Thus one can try to numerically compute the Lyapunov exponents for this discrete cocycle. In fact, this was how Zorich originally conjectured a simple spectrum for the case of Abelian differentials. We will not go into details behind the discrete theory of (half-)translation surfaces, that of interval exchange transformations, zippered rectangles, Rauzy-Veech induction, Zorich acceleration, generalized permutations, et cetera. We have written this section assuming the reader is acquainted with these concepts. We will give references for the unfamiliar but interested reader. 

The language of generalized permutations \cite{BoissyLanneau} is the right discrete language in which to study the dynamics of the discrete cocycle on a surface carrying a non-orientable quadratic differential. Not surprisingly, one can pass to the orienting double cover and study the dynamics of the Rauzy-Veech-Zorich cocycle for an interval exchange transformation through analogues of the already-developed tools for interval exchange transformations. The concept of \emph{interval exchange transformation with involution}, first introduced in \cite{AvilaResende}, is the right analogue of interval exchange transformations for Abelian differentials which are the pull-back of non-orientable ones. Although the explicit connection between generalized permutations and interval exchange transformations with involution, as well as explicit expressions for all the cocycles involved on the orienting cover, are not found in the literature, it is not hard to work them out from \cite{BoissyLanneau} and \cite{AvilaResende}. Having computed the matrix-valued cocycle expressions for the interval exchange transformations with involution, we have approximated the Lyapunov exponents for such cocycles numerically, following \cite[\S V.C]{eckmann-ruelle}. 

Below is a table of all the strata of quadratic differentials for which the Lyapunov exponents were approximated numerically. Recall that we always have $\lambda_1^- = 1$. According to \cite{lanneau:connected}, some strata are not connected and in some cases we have computed the exponents for different components of such strata. Note that the result for $\mathcal{Q}(2,-1,-1)$ has actually been proved in \cite[Theorem 1.7]{bainbridge}. The results for all strata examined suggest a simple spectrum, so we conjecture that this is true for $\hat{\mu}_\kappa$-almost all quadratic differentials for any singularity pattern $\kappa$.

\newpage

\begin{tabular}{|l|l|l|l|}
\hline
Stratum&Geni&Invariant Exponents& Anti-Invariant Exponents\\
\hline
$\mathcal{Q}(2,-1,-1)$&$g = 1,\, \hat{g} = 2$&$\lambda_1^+ = \frac{1}{2}$&$\lambda_1^- = 1$\\
\hline
$\mathcal{Q}(2,1,-1^3)$&$g = 1,\, \hat{g} = 3$&$\lambda_1^+ = \frac{1}{2}$&$\lambda_2^- = \frac{1}{3}$\\
\hline
&$g = 3,\, \hat{g} = 5$&$\lambda_1^+ = 0.660189$&$\lambda_2^- = 0.2000206$\\
$\mathcal{Q}(8)$&&$\lambda_2^+ = 0.3973745$& \\
&&$\lambda_3^+ = 0.142043$& \\
\hline 
&$g = 3,\,\hat{g} = 4$&$\lambda_1^+ = 0.778654$&$\lambda_2^- = 0.551526333$ \\
$\mathcal{Q}(-1,3,3,3)^{adj}$&&$\lambda_2^+ = 0.47222$& $\lambda_3^- = 0.233913333$ \\
&&$\lambda_3^+ = 0.229875$&$\lambda_4^- = 0.097543$ \\
\hline
&$g = 3,\,\hat{g} = 4$&$\lambda_1^+ = 0.597168$&$\lambda_2^- = 0.327950333$ \\
$\mathcal{Q}(-1,3,3,3)^{irr}$&&$\lambda_2^+ = 0.402619$& $\lambda_3^- = 0.190083$ \\
&&$\lambda_3^+ = 0.200314$&$\lambda_4^- = 0.083007333$ \\
\hline 
&$g = 3,\,\hat{g} = 3$&$\lambda_1^+ = 0.601297$&$\lambda_2^- = 0.30827666$ \\
$\mathcal{Q}(-1,3,6)^{adj}$&&$\lambda_2^+ = 0.3795885$& $\lambda_3^- = 0.1406165$ \\
&&$\lambda_3^+ = 0.1677125$&\\
\hline 
&$g = 3,\,\hat{g} = 3$&$\lambda_1^+ = 0.767285$&$\lambda_2^- = 0.524996$ \\
$\mathcal{Q}(-1,3,6)^{irr}$&&$\lambda_2^+ = 0.445894$& $\lambda_3^- = 0.17866075$ \\
&&$\lambda_3^+ = 0.190788$&\\
\hline
&$g = 3,\,\hat{g} = 3$&$\lambda_1^+ = 0.607201$&$\lambda_2^- = 0.281791$ \\
$\mathcal{Q}(-1,9)^{adj}$&&$\lambda_2^+ = 0.346005$& $\lambda_3^- = 0.080341$ \\
&&$\lambda_3^+ = 0.135734$&\\
\hline 
&$g = 3,\,\hat{g} = 3$&$\lambda_1^+ = 0.742725$&$\lambda_2^- = 0.4617$ \\
$\mathcal{Q}(-1,9)^{irr}$&&$\lambda_2^+ = 0.3902795$& $\lambda_3^- = 0.082813$ \\
&&$\lambda_3^+ = 0.139563$&\\
\hline
&$g = 4,\,\hat{g} = 7$&$\lambda_1^+ = 0.6639145$&$\lambda_2^- = 0.303482$\\
$\mathcal{Q}(12)^{I}$&&$\lambda_2^+ = 0.45256$&$\lambda_3^- = 0.119673$\\
&&$\lambda_3^+ = 0.2278785$& \\
&&$\lambda_4^+ = 0.089465$& \\
\hline
&$g = 4,\,\hat{g} = 7$&$\lambda_1^+ = 0.7476805$&$\lambda_2^- = 0.443258$\\
$\mathcal{Q}(12)^{II}$&&$\lambda_2^+ = 0.49137$&$\lambda_3^- = 0.12827975$\\
&&$\lambda_3^+ = 0.2437355$& \\
&&$\lambda_4^+ = 0.0893735$& \\
\hline
&$g = 3,\,\hat{g} = 2$&$\lambda_1^+ = 0.704425$&$\lambda_2^- = 0.33313725$ \\
$\mathcal{Q}(4,4)$&&$\lambda_2^+ = 0.4367675$& \\
&&$\lambda_3^+ = 0.1917245$&\\
\hline
$\mathcal{Q}(-1^2,1^2)$&$g = 1,\, \hat{g} = 3$&$\lambda_1^+ = \frac{2}{3}$&$\lambda_2^- = \frac{1}{3}$\\
\hline
$\mathcal{Q}(-1^3,1^3)$&$g = 1,\, \hat{g} = 4$&$\lambda_1^+ = 0.5449135$&$\lambda_2^- = 0.369280333$\\
&&&$\lambda_3^- = 0.176435$\\
\hline
$\mathcal{Q}(-1^4,1^4)$&$g = 1,\, \hat{g} = 5$&$\lambda_1^+ = 0.4768945$&$\lambda_2^- = 0.425535333$\\
&&&$\lambda_3^- = 0.261043333$\\
&&&$\lambda_4^- = 0.12274366$\\
\hline
$\mathcal{Q}(-1^5,5)$&$g = 1,\, \hat{g} = 4$&$\lambda_1^+ = 0.2841625$&$\lambda_2^- = 0.444530333$\\
&&&$\lambda_3^- = 0.1261055$ \\
\hline
$\mathcal{Q}(-1,2,3)$&$g = 2,\, \hat{g} = 4$&$\lambda_1^+ = 0.5829105$&$\lambda_2^- = 0.19990875$\\
&&$\lambda_2^+ = 0.3171165$&\\
\hline
\end{tabular}

\bibliographystyle{amsalpha}
\bibliography{biblio}
\end{document}